\documentclass[12pt]{amsart}
 \usepackage{amsmath}
 \usepackage{amssymb}
 \usepackage{amscd}
 \usepackage[mathscr]{eucal}
\usepackage{epsfig}
\usepackage{graphics}
 \newtheorem{thm}{Theorem}[section]
 \newtheorem{lem}[thm]{Lemma}
 
 \newtheorem{example}[thm]{Example}
 \newtheorem{cor}[thm]{Corollary}
 \newtheorem{prop}[thm]{Proposition}
 \newtheorem{defn}[thm]{Definition}
 
 \newtheorem{conj}[thm]{Conjecture}
 \newtheorem{prob}[thm]{Problem}

\newcommand\Prefix[3]{\vphantom{#3}#1#2#3}
\newcommand{\lk}{\operatorname{lk}}
\newcommand{\rk}{\operatorname{rk}}

\newcommand{\st}{\operatorname{st}}

\renewcommand{\aa}{\mathbf{a}}
\newcommand{\bb}{\mathbf{b}}
\newcommand{\cc}{\mathbf{c}}
\newcommand{\dd}{\mathbf{d}}

\newcommand{\Z}{\mathbb{Z}}
\newcommand{\Q}{\mathbb{Q}}
\newcommand{\N}{\mathbb{N}}

\newcommand{\R}{\mathbb{R}}
\newcommand{\C}{\mathbb{C}}
\newcommand{\tp}{\hat{1}}
\renewcommand{\th}{\hat{h}}
\newcommand{\tg}{\hat{g}}
\newcommand{\bt}{\hat{0}}
\newcommand{\Th}{\tilde{h}}
\newcommand{\PP}{\mathcal{P}}
\newcommand{\xx}{\mathbf{x}}

\begin{document}
 
 \title{Face enumeration - from spheres to manifolds}
\thanks{Partially supported by NSF grant DMS-0245623.}
\author{Ed Swartz}
\address{Dept. of Mathematics, Cornell University}
\address{Ithaca, NY 14853}
\email{ebs@math.cornell.edu}

\begin{abstract}
We prove a number of new restrictions on the enumerative properties  of homology manifolds and semi-Eulerian complexes and posets.  These include a determination of the affine span of the fine $h$-vector of balanced semi-Eulerian complexes  and the toric $h$-vector of semi-Eulerian posets.  

The lower bounds on simplicial  homology manifolds, when combined with  higher dimensional analogues of Walkup's 3-dimensional constructions \cite{Wal}, allow us to give a complete characterization of the $f$-vectors of arbitrary simplicial triangulations of $S^1 \times S^3, \C P^2,$ $ K3$ surfaces, and $(S^2 \times S^2) \# (S^2 \times S^2).$  We also establish a principle which leads to a conjecture for homology manifolds which is almost logically equivalent to the $g$-conjecture for homology spheres.  Lastly, we show that with sufficiently many vertices, every triangulable homology  manifold without boundary of dimension three or greater can be triangulated
in a $2$-neighborly fashion.

   \end{abstract}

  \maketitle
  
  \section{Introduction}
 
\ The fundamental {\it combinatorial} invariant of a $(d-1)$-dimensional triangulated space is its $f$-vector, $(f_0,\dots,f_{d-1}),$ where $f_i$ counts the number of $i$-dimensional faces.  After the Euler-Poincar\'{e} formula, the Dehn-Sommerville equations for simplicial polytopes are the best known restrictions on the $f$-vectors of manifolds.  While algebraic topology in general, and the topology of manifolds in particular,  made great strides in the first half of the twentieth century, it was not until 1964 that Klee published the manifold equivalent of the Dehn-Sommerville equations.

In the 70's, the introduction of commutative algebra in the form of the face ring by Hochster \cite[Theorems 4.1 and 4.8]{St}, Reisner \cite{Re} and Stanley \cite{St1}, and the connection between toric varieties and rational polytopes (see, for instance, \cite{Da}),  led to dramatic advances in the understanding of the enumerative properties of polytopes and spheres.  By 1980, McMullen's conjectured characterization of the $f$-vectors of simplicial convex polytopes \cite{Mc3} was verified by Stanley \cite{St2} (necessity), and Billera and Lee \cite{BLe} (sufficiency).  Since then, one of the most important problems in understanding the combinatorics of triangulations has become known as the $g$-conjecture (cf. Conjecture \ref{g-conj}): Do $f$-vectors of simplicial spheres, or more generally homology spheres, also satisfy McMullen's conditions?

Motivated by a desire to understand the face posets of polytopes, the 1980's and 90's saw the introduction of balanced complexes \cite{St5},  the $cd$-index \cite{BBa}, \cite{BK}, and the toric $h$-vector \cite[Section 3.14]{St6}.  All of these invariants make sense and were studied in the context of Eulerian posets, which include the face posets of regular cell decompositions of spheres and odd-dimensional compact manifolds without boundary.  Section \ref{complexes} is devoted to extending these ideas to semi-Eulerian posets and complexes.  These include the face posets of regular cell decompositions of compact even-dimensional manifolds without boundary.
The main results determine the affine span of each of these invariants.

The great variety and complications possible in the topology of manifolds has made the study of their $f$-vectors a daunting task.  At present there is not even a guess as to what the set of all possible $f$-vectors of manifolds (without boundary) would look like in dimensions greater than three.  The two most comprehensive conjectures in print are due to Kalai, \cite[Conjecture 7.5]{No} and K\"uhnel, \cite[Conjecture 18]{Lu}.  While these conjectures would have far reaching consequences for $f$-vectors of manifolds, they only concern the rational Betti numbers.  It is not an exaggeration to say that at this point there is no understanding whatsoever of the impact on the combinatorics of triangulations of many of the classical manifold invariants such as the cohomology ring structure, characteristic classes, or even torsion Betti numbers!  Perhaps it is appropriate that as of the beginning of the twenty-first century  it is still an open question in dimensions five and above  whether or not every compact topological manifold  without boundary  has a triangulation.  For information on what {\it is} known, especially concerning combinatorial manifolds, see the recent surveys by Datta \cite{Dat} and  Lutz \cite{Lu}.

One of the main results in Section \ref{inequalities}, Theorem \ref{g-thm to manifold}, can roughly be interpreted to mean that the distance between what we know about spheres and manifolds, while still substantial,  is not as great as it might seem.  It turns out that there is a conjecture for homology manifolds which is almost logically equivalent to the $g$-conjecture for spheres.  The rest of the section contains a number of restrictions on the $f$-vectors of homology manifolds.  All of our proofs work for arbitrary triangulations, not just combinatorial ones. The main new feature is the use of the face ring to produce lower bounds for the number of vertices {\it and edges}.  One consequence is  that  K\"uhnel's triangulations of sphere bundles over the circle \cite{Ku2} minimize the $f$-vector over all homology manifolds without boundary and nonzero first Betti number (Theorem \ref{min nonzero betti}).

The last section contains several constructions, most of which are higher dimensional analogues of those introduced by Walkup in dimension three \cite{Wal}.  In combination with our previous results, these techniques allow us to give complete characterizations of the $f$-vectors of $S^1 \times S^3, \C P^2,$ any K3 surface, and  $(S^2 \times S^2) \# (S^2 \times S^2).$   In addition, many partial results are possible, such as a description of all possible pairs $(f_0, f_1)$ which can occur in triangulations of $S^3 \times S^3.$   We end with another extension to higher dimensions of a result of Walkup's in dimension three. This theorem says that for any boundaryless homology manifold $M^{d-1}$  which can be triangulated, there exists $\gamma(M^{d-1})$ such that if $f_1 - d f_0 \ge \gamma(M^{d-1}),$ then there is a triangulation of $M^{d-1}$ with $f_0$ vertices  and $f_1$ edges.  In particular, for sufficiently many vertices, $M^{d-1}$ has a $2$-neighborly triangulation.

We have covered all of the manifolds for which we know necessary and sufficient conditions on the $f$-vectors of all possible triangulations.  Otherwise, we have not attempted to be encyclopedic in listing all possible applications of our methods to the large number of currently known triangulations. Rather, we have given a sample of the ways these techniques might be employed.   

Note:  Since this paper was originally written the set of all possible $f$-vectors of the nonorientable $S^3$ -bundle over $S^1$ was determined in \cite{CSS} using Theorem \ref{h by betti}.

\section{Notations and conventions}

Throughout, $\Delta$ is a connected, pure, $(d-1)$-dimensional simplicial complex with $n$ vertices and vertex set $V  = \{v_1, \dots, v_n\}.$ A simplicial complex is {\it pure} if all of its facets (maximal faces)  have the same dimension.  In addition, we will always  assume that $d \ge 4.$  The {\it geometric realization} of $\Delta, |\Delta|$ is the union in $\R^n$ over all faces $\{v_{i_1}, \dots, v_{i_j}\}$ of $\Delta$ of the convex hull of $\{e_{i_1},\dots,e_{i_j}\},$ where $\{e_1, \dots, e_n\}$ is the standard basis of $\R^n.$ We say $\Delta$ is {\it homeomorphic} to another space whenever $|\Delta|$ is.  A {\it triangulation} of a topological space $M$ is any simplicial complex $\Delta$ such that $\Delta$ is homeomorphic to $M.$  

The {\it link} of a face $\rho \in \Delta$ is  

$$\lk \rho=\displaystyle\bigcup_{\stackrel{\tau \cup \rho\in \Delta}{\tau \cap \rho= \emptyset}} \tau.$$

The {\it closed star} of a face $\rho \in \Delta$ is
$$\overline{\st} \rho= \displaystyle\bigcup_{\stackrel{\sigma \subseteq \tau, \tau \supseteq \rho}{\sigma \in \Delta}} \sigma.$$

The {\it join} of $\Delta$ and $\Delta^\prime,$ where the vertex set $V^\prime$ of $\Delta^\prime$ is disjoint from $V$, is

$$\Delta \ast \Delta^\prime = \{\rho \cup \rho^\prime: \rho \in \Delta, \rho^\prime \in \Delta^\prime\}.$$

For any poset $(P,\le),$ the {\it order complex} of $P$ is the simplicial complex whose  vertices are the elements of $P$ and whose faces are chains of $P.$ If $P$ contains a greatest element $\tp$ and/or a least element $\bt,$ then the {\it reduced} order complex of $P$ is the order complex of $P - \{\bt, \tp\}.$

Homology manifolds are a natural generalization of  topological manifolds.  Fix a field $k$. If for all $x \in |\Delta|,$ $\tilde{H}_i(|\Delta|, |\Delta|-x;k) = 0$ when $i < d-1,$ and either $k$ or $0$ when $i=d-1,$ then $\Delta$ is a {\it $k$-homology manifold.}  Equivalently, for every nonempty face $\rho \in \Delta,$ $H_\star(\lk \rho; k)$ is isomorphic to either the $k$-homology of $S^{d-|\rho|-1}$ or $B^{d-|\rho|-1},$ where $B^{d-|\rho|-1}$ is the $d-|\sigma|-1$-dimensional ball.  The {\it boundary} of a homology manifold, denoted $\partial \Delta,$ is the subcomplex consisting of all of the faces $\rho$ such that $H_{d-|\rho|-1}(\lk \rho;k) = 0.$ If  $H_{d-1}(\Delta, \partial \Delta;k) \simeq k,$ then $\Delta$ is {\it orientable} over $k.$  We say $\Delta$ is a {\it closed} homology  manifold over $k$  if $\Delta$ has  no boundary and is orientable over $k.$   If the boundary of $\Delta$ is not empty, then $\partial \Delta$ is a $(d-2)$-dimensional $k$-homology manifold without boundary \cite{Mi}.

The {\it $f$-vector} of $\Delta$ is $( f_0, \dots, f_{d-1}),$ where $f_i$ is the number of $i$-dimensional faces in $\Delta.$ Sometimes it is convenient to set $f_{-1} = 1$ corresponding to the emptyset.   The {\it face polynomial} of $\Delta$ is 

$$f_\Delta(x) = f_{-1} x^d + f_0 x^{d-1} + \dots + f_{d-2} x + f_{d-1}.$$

The {\it $h$-vector} of $\Delta$ is $(h_0, \dots, h_d)$ and is defined so that the corresponding $h$-polynomial, $h_\Delta(x) = h_0 x^d + h_1 x^{d-1} + \dots +h_{d-1} x +  h_d,$ satisfies $h_\Delta(x+1) = f_\Delta(x).$  Equivalently,

\begin{equation} \label{f by h}
  h_i = \sum^i_{j=0} (-1)^{i-j} \binom{d-j}{d-i} f_{j-1}.
\end{equation} 
 
Each $f_i$ is a  {\it nonnegative} linear combination of  $h_0, \dots, h_{i+1}.$  Specifically,
\begin{equation} \label{h by f}
  f_{i-1} = \sum^i_{j=0} \binom{d-j}{d-i} h_j.
\end{equation}

A simplicial complex $\Delta$ is {\it $i$-neighborly} if every subset of vertices of cardinality $i$ is a face of $\Delta.$  

A {\it stacked} polytope is the following inductively defined class of polytopes.   The simplex is a stacked polytope and any polytope obtained from a stacked polytope by adding a pyramid to a facet is a stacked polytope.  Stacked polytopes are simplicial and the boundary of a stacked polytope is a {\it stacked sphere}.  A purely combinatorial characterization of stacked spheres is due to Kalai.
 Let $\phi_i(n,d)$ be the number of $i$-dimensional faces in a $(d-1)$-dimensional stacked sphere with $n$ vertices.  Equivalently,

\begin{equation} \label{phi}
\phi_i(n,d) = \begin{cases}   \binom{d}{i} n - \binom{d+1}{i+1} i & \mbox{ for } 1 \le i \le d-2 \\  (d-1) n - (d+1)(d-2) & \mbox{ for } i=d-1. \end{cases}
\end{equation}

\begin{thm} \cite[Theorem 1.1]{Kal2} \label{Kal3}
  Let $\Delta$ be a homology manifold without boundary.  Then $f_i(\Delta) \ge \phi_i(n,d).$  If $f_i(\Delta) = \phi_i(n,d)$ for any $1 \le i \le d-1,$ then $\Delta$ is a stacked sphere.
\end{thm}

\begin{cor} \cite{Kal2}  \label{Kal2}
  Let $\Delta$ be a homology manifold without boundary.  Then $\Delta$ is a stacked sphere if and only if $h_1(\Delta) = h_2(\Delta).$
\end{cor}

Let $\Delta^\prime$ be another $(d-1)$-dimensional complex and let $\sigma^\prime$ be a facet of $\Delta^\prime.$  Let $\sigma$ be a facet of $\Delta$ and   choose a bijection between the vertices of $\sigma^\prime$ and the vertices of $\sigma.$  The {\it connected sum}  of $\Delta$ and $\Delta^\prime, \Delta \# \Delta^\prime,$ is the complex obtained by identifying the vertices (and corresponding faces) of $\Delta$ and $\Delta^\prime$  by the chosen bijection, and then removing the facet corresponding to $\sigma$ ($= \sigma^\prime).$  If both complexes are $(d-1)$-dimensional homology manifolds without boundary, then any connected sum  is also a homology manifold without boundary.  However, the homeomorphism type of $\Delta \# \Delta^\prime$ may depend on the chosen bijection.  Direct calculation shows that $h_d(\Delta \# \Delta^\prime) = h_d(\Delta) + h_d(\Delta^\prime) -1$ and for $0 < i < d, h_i(\Delta \# \Delta^\prime) = h_i(\Delta) + h_i(\Delta^\prime).$  

Another method for forming new complexes out of old is handle addition.  Let $\sigma$ and $\sigma^\prime$ be disjoint facets of $\Delta.$  Also, let $\phi$ be a bijection between the vertices of the two facets.  Identify each pair of vertices $(v, \phi(v))$ and any corresponding faces.  As long as $v$ and $\phi(v)$ are not neighbors and there are no vertices which have both $v$ and $\phi(v)$ as neighbors, the resulting space will still be a simplicial complex and we say it is obtained by {\it handle addition}.  If the original complex is a homology manifold without boundary, then so is the new complex.  As before, the homeomorphism type of the new complex may depend on the choice of bijection.

In \cite{Wal} Walkup introduced $\mathcal{H}^{d-1},$ the set of simplicial complexes that can be obtained from $(d-1)$-dimensional stacked spheres by repeated handle addition.  As we will see (cf. Theorem \ref{min h} and Theorem \ref{h by betti}), the triangulations in $\mathcal{H}^{d-1}$ are minimal in a certain sense.

\begin{thm} \label{walk1} \cite{Wal}, \cite{Kal2}
  $\Delta \in \mathcal{H}^{d-1}$ if and only if the link of every vertex of $\Delta$ is a stacked sphere.
\end{thm}

  \section{Linear relations}  \label{complexes}

   After the Euler-Poincar\'{e} formula, the Dehn-Sommerville equations for simplicial polytopes (\cite{De}, \cite{So}) were one of the first known restrictions on the $f$-vectors of a class of manifolds.  These relations for polytopes were generalized to semi-Eulerian complexes by Klee.  We say $\Delta$ is a {\it semi-Eulerian} complex if for every nonempty face $\rho$ of $\Delta,$ the Euler characteristic of its link, $\chi (\lk \rho),$ equals $\chi (S^{d - |\rho|-1}).$ Homology manifolds without boundary are a motivating example.   If in addition, $\chi (\Delta) = \chi( S^{d-1}),$ then we say $\Delta$ is an {\it Eulerian} complex.  
    
    \begin{thm}\cite{Ke}
      Let $\Delta$ be a  semi-Eulerian complex.  Then
        \begin{equation} \label{klee}
        h_{d-i} - h_i =  (-1)^i \binom{d}{i} (\chi (\Delta) - \chi (S^{d-1})).
       \end{equation}
    \end{thm}
    
\noindent  Our semi-Eulerian complexes were called Eulerian manifolds in \cite{Ke}.  Related equations were discovered earlier by Vaccaro \cite{Va}.

  If $\Delta$ is an odd-dimensional semi-Eulerian complex, then setting $i = d/2$ shows that the Euler characteristic of $\Delta$ is zero and hence $\Delta$ is Eulerian.    We will refer to the above equations as the {\it generalized Dehn-Sommerville equations}.    
    
    Under certain conditions there is a refinement of the generalized Dehn-Sommerville equations.  Let $\aa=(a_1,\dots,a_m)$ be a sequence of positive  integers. Define $|\aa| = a_1 + \dots + a_m.$  Let $\phi:V \to [m]$, with $[m]=\{1,\dots, m\}$, be a surjective function and set $V_j = \phi^{-1}(j).$
    
    \begin{defn} Suppose $|\aa|=d.$  The pair $(\Delta, \phi)$ is a {\bf balanced complex of  type} $\aa$ if for every facet $\sigma \in \Delta$ and $j, 1 \le j \le m,$
    $$|\sigma \cap V_j| = a_j.$$
   \end{defn}
   
\noindent   Balanced complexes of type $(1,\dots,1)$ are called {\it completely balanced.}  The canonical example of a completely balanced complex is the order complex of a graded poset with $\phi(v) = \rk (v).$  If $\bb = (b_1, \dots, b_m)$ is a sequence of nonnegative integers such that $b_j \le a_j,$ then we write $\bb \le \aa.$ When $\Delta$ is completely balanced we can identify sequences $\bb \le (1, \dots, 1)$ with subsets of $[d]$ in the usual way, $\bb \leftrightarrow \{i \in [d]: b_i = 1\}.$

One way to produce examples of balanced complexes is to start with a completely balanced complex $(\Delta,\phi)$ and specialize.   Given $\aa$ with $|\aa|=d,$   let $\psi: [d] \to [m]$ be the map such that $\psi^{-1}(j) = [a_1 + \dots + a_{j-1} + 1, a_1 + \dots + a_j].$  Then $(\Delta,\psi \circ \phi)$ is a balanced complex of type $\aa.$  Under these conditions we write $S \to \bb$ if $S \subseteq [d]$ and $|S \cap  [a_1 + \dots + a_{j-1} + 1, a_1 + \dots + a_j]| = b_j$ for each $j.$  If $\Delta$ is any $(d-1)$-dimensional pure complex and $|\aa|=d,$ then we can construct a balanced complex of type $\aa$ which is homeomorphic to $\Delta$.  Indeed, the (reduced) order complex of the face poset of $\Delta$ is a completely balanced complex homeomorphic to $\Delta$ which can then be specialized to $\aa.$  For a simple example of a balanced complex which is not the specialization of a completely balanced complex, see Figure \ref{bipyramid}.

Let  $(\Delta, \phi)$ be a balanced complex of type $\aa.$  For $\bb \le \aa$  define
$f_\bb$ (if necessary, $f_\bb(\Delta)$) to be the number of faces   $\rho $ such that for all $j, |\rho \cap \phi^{-1}(j)| = b_j.$  The collection  $\{f_\bb\}_{\bb \le \aa}$ is the {\it fine} $f$-vector of $\Delta$ and is a refinement of the $f$-vector in the sense that

$$\displaystyle\sum_{|\bb| = i} f_\bb = f_{i-1}(\Delta).$$

The {\it fine} $h$-vector of $\Delta$  is defined by

\begin{equation} \label{bal h}
h_\bb = \displaystyle\sum_{\cc \le \bb} f_\cc \ \displaystyle\prod^m_{i=1} (-1)^{b_i-c_i} \binom{a_i-c_i}{b_i-c_i}.
\end{equation}

\begin{example}
  The bipyramid in Figure \ref{bipyramid} is a balanced complex of type $(1,2).$ The fine $f$- and $h$-vectors are
$$\begin{array}{lclclcl}
f_{0,0} & = & 1 &\mbox{   }& h_{0,0} & = & 1\\
f_{1,0} & = & 2 &\mbox{   }& h_{1,0} & = & 1 \\ 
f_{0,1} & = & 5 & \mbox{   } & h_{0,1} & = & 3 \\
f_{1,1} & = & 10 & \mbox{   } & h_{1,1} & = & 3\\
f_{0,2} & = & 5 & \mbox{   } & h_{0,2} & = & 1\\
f_{1,2} & = & 10 & \mbox{   } & h_{1,2} & = & 1
\end{array}$$
\end{example}

\begin{thm} \cite{St5}
   Let $(\Delta, \phi)$ be a balanced complex of type $\aa.$  Then
   
   \begin{equation} 
   h_i = \displaystyle\sum_{|\bb|=i} h_\bb.
   \end{equation}
 \end{thm}
 
  If $\Delta$ is completely balanced and we have identified $\bb \le (1,\dots, 1)$ with subsets of $[d]$ as above, then the collections $f_S$ and $h_S, S \subseteq [d],$ are called the {\it flag} $f$-vector and {\it flag} $h$-vector respectively. Here $f_S$ is the number of faces $\rho$ such that the image of the vertices of $\rho$ under $\phi$ is $S.$  In this case  Equation (\ref{bal h}) becomes
 
 \begin{equation} \label{comp bal h}
    h_S = \displaystyle\sum_{T \subseteq S} (-1)^{|S-T|} f_S.
  \end{equation}  
    
  An equivalent way to define  $h$ and $h_\bb$ is through the  face ring.  Let $k$ be a field and set $R = k[x_1,\dots,x_n].$  
  
\begin{defn}
  The {\bf face ring} (also known as the {\bf Stanley-Reisner ring}) of $\Delta$ is $k[\Delta] = R/I_\Delta,$ where
  
  $$I_\Delta = <\{x_{i_1}, \dots, x_{i_k}: \{v_{i_1}, \dots, v_{i_k}\} \notin \Delta \}>.$$
\end{defn}  
    
   The Hilbert function of $k[\Delta]$ encodes the $h$-vector of $\Delta$ is a nice way. Let $k[\Delta]_i$ be the degree $i$ component of $k[\Delta].$  Define $$F(\Delta,\lambda) = \displaystyle\sum^\infty_{i=0} \dim_k k[\Delta]_i \ \lambda^i.$$
   
   \begin{thm} (See, for instance, \cite[II.2]{St} )
 $$ 
     F(\Delta,\lambda) = \displaystyle\sum^d_{i=0} \frac{h_i \lambda^i}{(1-\lambda)^d}.
$$
   \end{thm}
  
  When $\Delta$ is balanced, $k[\Delta]$ has a natural  $\N^m$ grading by assigning $x_i$ to $\lambda_{\phi(v_i)}.$ 
  For instance, let $\Delta$ be the boundary of a bipyramid over a pentagon in $\R^3$ as in Figure \ref{bipyramid}.   With $\phi$ as given, $\Delta$ is a balanced complex of type $\{1,2\}.$  The fine Hilbert function for $k[\Delta]$ is
  
  $$  1 + 2\lambda_1 +5 \lambda_2  + 2 \lambda^2_1 + 10 \lambda^2_2  + 10 \lambda_1 \lambda_2  + 2 \lambda^3_1 + 10 \lambda^2_1 \lambda_2 + 20 \lambda_1 \lambda^2_2 + 15 \lambda^3_2 + \dots.$$ 
  
  \begin{figure} 
 \scalebox{0.70}[0.70]{\includegraphics{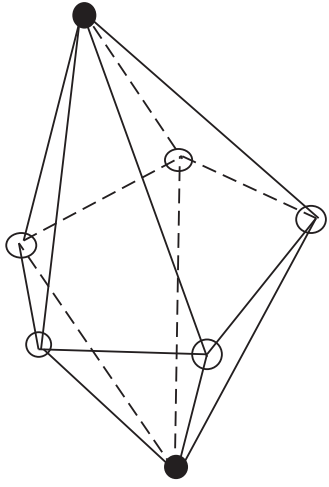}}
  \caption{Balanced bipyramid, $\phi(\bullet) = 1, \phi({\circ})=2.$} \label{bipyramid}
\end{figure}

    \begin{thm} \cite{St5}
      Let $\Delta$ be a balanced complex of type $\aa = (a_1,\dots,a_m).$  Then
      
      \begin{equation} \label{fine hilbert}
      F(k[\Delta],\lambda) =\prod^m_{j=1}  \frac{1}{(1-\lambda_j)^{a_j}} \displaystyle\sum_{\bb \le \aa} h_\bb \lambda^\bb.
      \end{equation}
      
  \end{thm}
  
  As usual $\lambda^\bb = \lambda^{b_1}_1 \cdots \lambda^{b_m}_m.$   When $\bb \le \aa,$ denote by $\aa - \bb$ the $m$-tuple $(a_1 - b_1, \dots, a_m -b_m).$
  For the special case of balanced Eulerian complexes, the following theorem was first stated in \cite{BMi}.
  
  \begin{thm}  \label{fine DS thm}
    If $\Delta$ is a balanced semi-Eulerian complex of type $\aa,$ then for all $\bb \le \aa,$
    
    \begin{equation} \label{fine DS}
    h_{\aa - \bb} - h_\bb = (-1)^{|b|} [\chi(\Delta) - \chi(S^{d-1})] \prod^m_{j=1} \binom{a_j}{b_j}.
    \end{equation}
  
  \end{thm}
  
  \begin{proof}
    The strategy of the proof is not new and follows the ideas of  \cite[II.7]{St}.  We compute the fine Hilbert function $F(k[\Delta],1/\lambda)$ in two different ways.  From equation (\ref{fine hilbert}),
    
    $$\begin{array}{ccl}
     F(k[\Delta],1/\lambda) & = &\displaystyle\prod^m_{j=1}  \frac{1}{(1-1/\lambda_j)^{a_j}} \displaystyle\sum_{\bb \le \aa} h_\bb /\lambda^\bb \\
     & =  & \displaystyle\prod^m_{j=1}  \frac{1}{(\lambda_j-1)^{a_j}} \displaystyle\sum_{\bb \le \aa} h_\bb \lambda^{\aa-\bb} \\ 
     & = & (-1)^d \displaystyle\prod^m_{j=1}  \frac{1}{(1-\lambda_j)^{a_j}} \displaystyle\sum_{\bb \le \aa} h_\bb \lambda^{\aa-\bb}.
     \end{array}$$
     
    For a face $\rho,$ define $\rho(l) = |v \in \rho: \phi(v) = l|.$ By \cite[Corollary 7.2]{St},
     
      $$\begin{array}{ccl}
  (-1)^d   F(k[\Delta],1/\lambda) & = & (-1)^{d-1} \tilde{\chi}(\Delta) + \displaystyle\sum_{\rho \neq \emptyset} \displaystyle\prod_{v_i \in \rho} \frac{\lambda_{\phi(v_i)}}{1 - \lambda_{\phi(v_i)}} \\
    &  =  & (-1)^{d-1} \tilde{\chi}(\Delta) + \\
    & & \displaystyle\prod^m_{j=1}  \frac{1}{(1-\lambda_j)^{a_j}} \displaystyle\sum_{\rho \neq \emptyset} \displaystyle\prod_{v_i \in \rho} \lambda_{\phi(v_i)} \displaystyle\prod^m_{l=1} (1 - \lambda_l)^{a_l-\rho(l)}  \\ 
     &  =  & (-1)^{d-1} \tilde{\chi}(\Delta) + \\
     & &  \displaystyle\prod^m_{j=1}  \frac{1}{(1-\lambda_j)^{a_j}} \displaystyle\sum_{\bb \leq \aa} \displaystyle\sum_{\stackrel{\cc \leq \bb}{|\cc| \neq 0}}  (-1)^{|\bb-\cc|} f_\cc \displaystyle\prod^m_{l=1} \binom{a_l - c_l}{b_l-c_l} \lambda^\bb \\
     & = & (-1)^{d-1} \tilde{\chi}(\Delta) +\\
    & & \displaystyle\prod^m_{j=1}  \frac{1}{(1-\lambda_j)^{a_j}} \left\{ \displaystyle\sum_{\bb \leq \aa} h_\bb - (-1)^{|\bb|} \displaystyle\prod^m_{l=1} \binom{a_l}{b_l} \right\} \lambda^\bb.
     \end{array}$$

Multiplying both equations by  $ \displaystyle\prod^m_{j=1}  \frac{1}{(1-\lambda_j)^{a_j}}$ leaves

$$\displaystyle\sum_{\bb \le \aa} h_\bb \lambda^{\aa-\bb} = \displaystyle\sum_{\bb \le \aa} \left\{ h_\bb + (-1)^{|\bb|} [(-1)^{d-1} \tilde{\chi}(\Delta) -1] \displaystyle\prod^m_{j=1} \binom{a_j}{b_j} \right\} \lambda^\bb.$$

Since $(-1)^{d-1} \tilde{\chi}(\Delta) -1 = \chi(\Delta) - \chi(S^{d-1}),$ comparing the coefficients of $\lambda^\bb$ finishes the proof.
    
      \end{proof}
  
 As far as we know, the only other place that semi-Eulerian balanced (as opposed to completely balanced) complexes are considered is Magurn \cite{Mag}, where balanced compact 2-manifolds are analyzed.  Equation (\ref{fine DS}) for completely balanced semi-Eulerian posets appears in \cite[Proposition 2.2]{St9}. Balanced complexes of type $\aa=(d)$ are just pure complexes, and in this case (\ref{fine DS}) recovers the generalized Dehn-Sommerville equations.
 
 \begin{cor}
   If $\Delta$ is a completely balanced semi-Eulerian complex, then
   
   \begin{equation} \label{flag h DS}
   h_{[d]-S} - h_S = (-1)^{|S|} [\chi(\Delta) - \chi(S^{d-1})].
   \end{equation}
 \end{cor}

 For Eulerian complexes the  relations in the above corollary are also called the generalized Dehn-Sommerville equations. For the history of these equations see the discussion in \cite{BBa}.

 Let $H_E(d)$ be the  affine span of fine $h$-vectors of  balanced $(d-1)$-dimensional Eulerian  complexes of type $\aa$.  Billera and Magurn determined the dimension of $H_E(d)$ in \cite{BMi}.  Their answer was in terms of the number of $\bb \le \aa, n(\aa) = \displaystyle\prod^m_{j=1}  (a_j + 1).$  Equation (\ref{fine DS}) allows us to extend their result to semi-Eulerian complexes.  
  
    \begin{thm} \label{fine affine span}
     Let $\Delta$ be a  semi-Eulerian complex.  Fix $\aa, |\aa|=d.$  Let $H_\Delta$ be the affine span of $\{h_\bb(\Delta^\prime)\}$, where $\Delta^\prime$ ranges over all balanced complexes  of type $\aa$ homeomorphic to $\Delta.$ Then 
     \begin{equation} \label{affine span fine DS}
     \dim H_\Delta =
          \begin{cases}
         	\frac{1}{2} (n(\aa)-1) 	& \mbox{ if every } a_i \mbox{ is even, } \\
	      	\frac{1}{2} (n(\aa)-2)	& \mbox{ otherwise. } 
	\end{cases}
     \end{equation}
	
     \end{thm}
     
     \begin{proof}
        If $\bb \le \aa,$ then $\bb \neq \aa - \bb$ unless each $a_i$ is even and $b_i = a_i/2$ for every $i.$ Also, $h_{\{0,\dots,0\}} = 1$ for any balanced complex.  Hence, Theorem \ref{fine DS thm} implies that $H_\Delta$ satisfies $1+ \frac{n(\aa)-1}{2}$ linearly independent equations if every $a_i$ is even, and $1+\frac{n(\aa)}{2}$  otherwise. Therefore, the dimension of $H_\Delta$ is bounded above by the right-hand side of (\ref{affine span fine DS}). 
        
         In order to prove the opposite inequality, we first  construct the requisite number of balanced spheres of type $\aa$ whose fine $h$-vectors affinely span $H_\Delta$   for $\Delta = S^{d-1}.$   This is accomplished in \cite[Section 5]{BMi}.  Denote by $\{\PP^\cc\}_{\cc \in \mathcal{C}}$ the corresponding collection of balanced spheres.

     Now let $\Delta$ be an arbitrary semi-Eulerian complex.  As noted before, there exists $\Delta^\prime$ homeomorphic to $\Delta$ with $\Delta^\prime$ a balanced complex of type $\aa.$  
  For  $\cc \in \mathcal{C}, \bb \neq (0, \dots, 0), \bb \neq \aa, h_\bb(\Delta^\prime \#  \PP^\cc) = h_\bb(\Delta^\prime) + h_\bb( \PP^\cc),$ where the connected sum identifies vertices so that the resulting complex is still balanced of type $\aa.$     In addition, $h_{\{0,\dots,0\}}(\Delta^\prime \# \PP^\cc) = 1$ and $h_\aa(\Delta^\prime \#  \PP^\cc) = h_\aa(\Delta^\prime).$  So, the affine span of $\{h_\bb(\Delta^\prime \#  \PP^\cc)\}_{\cc \in \mathcal{C}}$ is a translation of the affine span of $\{h_\bb( \PP^\cc)\}_{\cc \in \mathcal{C}}$ and hence has the same dimension.
     \end{proof}

     When $\Delta$ is the order complex of a poset there are further restrictions on the flag $h$-vector of $\Delta.$ A finite graded poset $P$ with a least element $\hat{0}$ and greatest element $\hat{1}$   is semi-Eulerian if $\mu(x,y) = (-1)^{\rk(x) - \rk(y)}$ for all $x \le y, (x,y) \neq (\hat{0},\hat{1}).$  Equivalently, the reduced order complex of $P$ is a semi-Eulerian complex.  If in addition, $\mu(\hat{0},\hat{1}) = (-1)^{\rk(P)},$ then we say $P$ is {\it Eulerian.}
     
    Let $F_E(d)$ be the affine span of flag $f$-vectors of reduced order complexes of rank $d$ Eulerian posets. Bayer and Billera determined $F_E(d)$ explicitly \cite{BBa}.  For the purposes of stating the linear equations satisfied by the elements of $F_E(d),$ we temporarily extend the definition of $f_S$ to subsets $S \subseteq \{0,1,\dots,d-1\}.$   If $0 \in S,$ then define $f_S = f_{S-\{0\}}.$      
     
     \begin{thm} \cite{BBa}
       Let $P$ be an Eulerian poset of rank $d,$  let $\Delta$ be the order complex of $P$ and let $S \subseteq [d-1].$ If $\{i,k\} \subseteq S \cup \{-1,d\}, i<k-1,$ and $S$ contains no $j$ such that $i<j<k,$ then
       \begin{equation} \label{cd inequalities}
         \displaystyle\sum^{k-1}_{j=i+1} (-1)^{j-i-1} f_{S \cup j} = f_S (1-(-1)^{k-i-1}).
       \end{equation}
         \end{thm}
         
     Bayer and Billera proved that the affine span of the set of flag $f$-vectors which satisfy (\ref{cd inequalities})   has dimension $e_d -1, $ where $e_d$ is the $d$-th Fibonacci number.  Then they constructed a family, $\PP^d,$ of polytopes whose flag $f$-vectors were affinely independent with $|\PP^d| = e_d, $ thus proving that $F_E(d)$ consists of  all $\{f_S\}$ which satisfy (\ref{cd inequalities}).
     
    J.  Fine gave a basis for $F_E(d)$ which we now describe.  The coefficients with respect to this basis have come to be known as the $\cc \dd$-index of $P.$  Encode the flag $h$-vector of $P$ (or more accurately, of the reduced order complex of $P$) as a polynomial $h_P(\aa,\bb)$ in noncommuting variables $a$ and $b$ (not to be confused with the indices in the previous section) by
     
     $$S \leftrightarrow \begin{cases} \aa,& i \notin S \\ \bb, & i \in S. \end{cases}$$
\noindent For instance, if $P$ is the face poset of the bipyramid in Figure \ref{bipyramid}, then 
$$h_P(\aa,\bb) = \aa \aa \aa + 6 \bb \aa \aa + 14 \aa \bb \aa + 9 \aa \aa \bb + 6 \aa \bb \bb + 14 \bb \aa \bb + 9 \bb \bb \aa + \bb \bb \bb.$$ 

Now let  $\cc = \aa+\bb$ and $\dd = \aa \bb + \bb \aa.$ In the above example, $h_P(\aa, \bb) = \cc \cc \cc + 5 \dd \cc + 8 \cc \dd.$  Let $F(\cc,\dd)$ be the linear subspace spanned by all monomials in $\cc$ and $\dd$ of degree  $d-1$ other than $\cc^{d-1},$ where the degree of $\cc$ is one and the degree of $\dd$ is two. Bayer and Klapper proved that $F_E(d) = \cc^{d-1} + F(\cc,\dd)$ \cite[Theorem 4]{BK}.

The results for flag $f$-vectors of semi-Eulerian posets are similar.  Suppose $P$ is a rank $d$ semi-Eulerian poset and let $\Delta_P$ be the reduced order complex of $P.$ Note that the dimension of $\Delta_P$ is $d-2.$   In order to describe $F_P,$ the affine span of flag $f$-vectors of posets whose order complexes are homeomorphic to $\Delta_P,$ set $X = \chi(\Delta_P) - \chi(S^{d-2}).$  

\begin{thm}  \label{semi cd}
  Let $f^X(d)$ be the flag $f$-vector which is zero for all $S \subseteq [d-1]$ except $f_{\{d-1\}} = X.$  Then $ F_P = F_E(d) + f^X(d).$
\end{thm}

\begin{proof}
    If $d$ is odd, then $P$ is Eulerian and $X=0$. Hence, $F_P \subseteq F_E(d).$  So assume that $d$ is even. Since $P$ is semi-Eulerian the flag $f$-vector of $\Delta_P$ satisfies all of the equations in  (\ref{cd inequalities}) except the one equivalent to Euler's formula for the sphere,
    $$\displaystyle\sum^{d-1}_{j=0} (-1)^j f_{\{j\}} = f_\emptyset (1-(-1)^{d-1}).$$
    
    Let $f^\prime$ be the flag $f$-vector defined by $f^\prime_S = f_S(\Delta_P) - f^X_S(d).$  The only expressions of (\ref{cd inequalities}) which are different for $f^\prime$ are the one above, which by the definition of $X$ is now valid, and
    $$\displaystyle\sum^{d-2}_{j=0} (-1)^j f^\prime_{\{d-1\} \cup j} = f^\prime_{\{d-1\}} (1-(-1)^d).$$
    Comparing this expression with the corresponding expression for $f$, the left hand side is unchanged, while the right hand side in both cases is zero since $d$ is even.  Hence $f^\prime$ is in $F_E(d).$  Therefore, $F_P \subseteq F_E(d) + f^X(d).$

    To establish the opposite inclusion, consider the family of flag $f$-vectors given by the (reduced) order complex of $\{\Delta_P \# \partial \PP_t\},$ where $\PP_t$ is the collection of $(d-1)$-polytopes given by Bayer and Billera whose flag $f$-vectors (affinely) span $F_E(d)$ \cite{BBa}.  As this set of flag $f$-vectors is a translation of the flag $f$-vectors of the reduced order complex of $\{\partial \PP_t\},$ its affine dimension is the same.  Since each $\Delta_P \# \partial \PP_t$ is homeomorphic to $\Delta_P, \dim F_P \ge \dim F_E(d).$  
    \end{proof}
    
\noindent The containment $F_P \subseteq F_E(d) + f^X(d)$ is a special case of \cite[Theorem 4.2]{Eh}, where Ehrenborg considers posets whose intervals of varying lengths are Eulerian.  
    
 In view of Karu's proof that the {\it cd}-index of any Gorenstein* poset has nonnegative coefficients \cite{Kar}, and the fact that  the flag $f$-vectors of semi-Eulerian posets with the same Euler characteristic and dimension lie in the same affine subspace of flag $f$-vectors,  it seems natural to ask the following question.
 
 \begin{prob}
   For a fixed semi-Eulerian poset $P,$ describe the cone of flag $f$-vectors of posets $P^\prime$ such that $\Delta_P$ is homeomorphic to $\Delta_{P^\prime}.$ 
 \end{prob}  
 
  An alternative approach to the combinatorics of semi-Eulerian posets is through the toric $h$-vector.  Originally introduced to correspond to the Betti numbers of the intersection cohomology of toric varieties associated to rational polytopes, the toric $h$-vector can be defined for any  finite graded poset with a minimum element $\bt$ and a maximum element $\tp.$  With the exception noted below, we follow Stanley's presentation \cite[Section 3.14]{St6} and refer the reader to \cite{St7} for background on the motivation and history behind its definition.
     
Let $P$ be a finite graded poset with $\bt$ and $\tp$ and let $\rho$ be the rank function of $P.$   Let $\tilde{P}$ be the set of all intervals $[\bt,z]$ in $P$  ordered by inclusion.  The map $z \to [\bt,z]$ is a poset isomorphism from $P$ to $\tilde{P}.$  Define two polynomials $\tilde{h}$ and $\tilde{g}$ inductively as follows.

\begin{itemize}
  \item
    $\th(\mathbf{1},x) = \tg(\mathbf{1},x) = 1.$  Here $\mathbf{1}$ is the poset with only one element $\tp=\bt.$   \item
   If the rank of $P$ is $d+1,$ then $\th(P,x)$ has degree $d.$ Write $\th(P,x) = \th_d + \th_{d-1} x + \th_{d-2} x^2 + \dots + \th_0 x^d.$ Then define
    $\tg(P,x) = \th_d + (\th_{d-1} - \th_d) x + (\th_{d-2} - \th_{d-1}) x^2 + \dots + (\th_{d-m} - \th_{d-m+1}) x^m,$
    where $m = \lfloor d/2 \rfloor.$
    
    \noindent [NOTE: Our $\th_i$ is $\th_{d-i}$ in \cite{St6}.]
  \item
  If the rank of $P$ is $d+1,$ then define
  \begin{equation} \label{toric def} 
   \th(P,x) = \displaystyle\sum_{\stackrel{Q \in \tilde{P}}{Q \neq P}} \tg (Q,x) (x-1)^{d - \rho(Q)}.
  \end{equation}
  
\end{itemize}

Induction shows that if $B_d$ is the face poset of the $(d-1)$-simplex, then $\th(B_d,x) = 1 + x + \dots + x^{d-1}$ and $\tg(B_d,x) = 1.$  From this it follows that if $P$ is the face poset of a simplicial complex $\Delta$ with $\tp$ adjoined, then $\th_i(P,x) = h_i(\Delta).$  

\begin{example}
Figure \ref{poset} shows the Hasse digram, $P$, of the face poset of the cell decomposition of the torus depicted in Figure \ref{torus}.   The rank one and rank two elements correspond to simplices.  The four rank two elements, which correspond to the four rectangles of the cell decomposition have $\hat{g} = 1 + x.$  So,

$$\th_P(x) = (x-1)^3 + 4(x-1)^2 + 8(x-1) + 4(x+1) = x^3 + x^2 + 7x -1.$$
\end{example}

\begin{figure} 
 \scalebox{0.60}[0.60]{\includegraphics{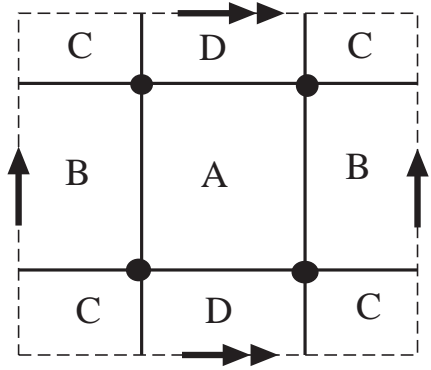}}
  \caption{Cell decomposition of the torus} \label{torus}
\end{figure}

\begin{figure} 
 \scalebox{0.40}[0.40]{\includegraphics{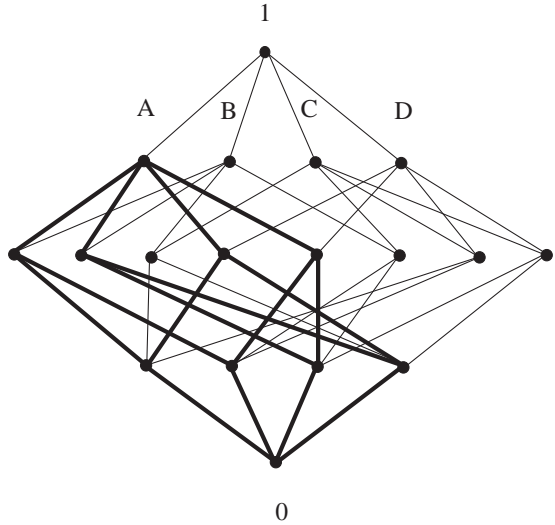}}
  \caption{Hasse diagram of $P$.} \label{poset}
\end{figure}

\begin{thm} \label{toric h}
  Let $P$ be a semi-Eulerian poset of rank $d+1$ and let $\Delta_P$ be the reduced order complex of $P.$ Then
  \begin{equation} \label{toric formula}
  \th_{d-i} - \th_i = (-1)^i \binom{d}{i} [\chi(\Delta_P) - \chi(S^{d-1})].
  \end{equation}
 \end{thm}

\begin{proof}
  The proof is a  small variation of Stanley's proof of this equation in the special case that $P$ is Eulerian \cite[pg. 139]{St6}.  Write $\th(P)$ for $\th(P,x)$ and $\tg(P)$ for $\tg(P,x).$  Let $y = x-1.$ Mutliply (\ref{toric def}) by $y$ and add $\tg(P)$ to obtain for $P \neq \mathbf{1},$
  
  $$\tg(P) + y \ \th(P) = \displaystyle\sum_{Q \in \tilde{P}} g(Q) y^{\rho(P) - \rho(Q)}.$$ 
  
\noindent Hence for $P \neq \mathbf{1},$

$$y^{-\rho(P)}(\tg(P) + y\ \th(P)) = \displaystyle\sum_Q \tg(Q) y^{-\rho(Q)}.$$

Since $\sum_{Q \in \mathbf{1}} \tg(Q)y^{-\rho(Q)} = 1,$ M\"{o}bius inversion implies,

$$\tg(P) y^{-\rho(P)} = \mu_P(\bt,\tp) + \displaystyle\sum_{\stackrel{Q \in \tilde{P}}{Q \neq \mathbf{1}}} (\tg(Q) + y\ \th(Q)) y^{-\rho(Q)} \mu_{\tilde{P}} (Q,P) $$ 

Since $\tilde{P}$ is semi-Eulerian, $\mu_{\tilde{P}} (Q,P) = (-1)^{\rho(P) - \rho(Q)}.$ So,

\begin{equation} \label{toric1}
  \tg(P) = y^{\rho(P)} \mu_P(\bt,\tp) + \displaystyle\sum_{Q \neq \mathbf{1}} (\tg(Q) + y \th(Q)) (-y)^{\rho(P) - \rho(Q)}.
\end{equation}

Let $\th(Q) = a_0 + a_1 x + \dots + a_r x^r,$ where $\rho(Q) = r+1.$  Then

$$\tg(Q) + y \th(Q) = (a_s - a_{s+1}) x^{s+1} + (a_{s+1} - a_{s+2}) x^{s+2} + \dots,$$

\noindent where $s = \lfloor r/2 \rfloor.$  Since each $Q$ is neither  $\mathbf{1}$ nor $P,$ it is Eulerian, so we may assume by induction on the rank of $Q$ that $a_i = a_{r-i},$ where $r +1 = \rho(Q), r<d.$  In this case,

\begin{equation} \label{toric2}
  \begin{array}{lcl} \tg(Q) + y \ \th(Q) & = & (a_s - a_{s-1}) x^{s+1} + (a_{s-1} - a_{s-2}) x^{s+2} +\dots \\
  \ & = & x^{\rho(Q)} \tg(Q, 1/x).
  \end{array}
 \end{equation}
  
  Now subtract $y \ \th(P) + \tg(P)$ from both sides of (\ref{toric1}) and use (\ref{toric2}) to obtain

$$\begin{array}{lcl}
-y \th(P) &  = & y^{\rho(P)} \mu_P(\bt,\tp) + \displaystyle\sum_{\bt< Q < \tp} x^{\rho(Q)} \tg(Q,1/x) (-y)^{\rho(P)-\rho(Q)} \\
\Rightarrow \th(P) &  = &- (y^d) [\mu_P(\bt,\tp) -(-1)^{d+1}] + \displaystyle\sum_{Q<\tp} x^{\rho(Q)} \tg(Q,1/x) (-y)^{d - \rho(Q)} \\
\ & = & -(y^d) [\mu_P(\bt,\tp) -(-1)^{d+1}] + x^d \th(P,1/x). \end{array}$$

Comparing like terms of the last equation gives

\begin{equation} \label{toric3}
\th_{d-i} - \th_i = (-1)^{d-i-1} \binom{d}{i}[\mu_P(\bt,\tp) - (-1)^{d+1}].
\end{equation}

When $d$ is even $P$ is Eulerian, so the right hand side of (\ref{toric3}) is zero and the equality agrees with (\ref{toric formula}).  If $d$ is odd, then, since $\mu_P({\bt,\tp}) = \chi(\Delta_P)-1$ and $(-1)^{d+1} = \chi(S^{d-1})-1,$ (\ref{toric3}) also agrees with (\ref{toric formula}).

\end{proof}

As the toric $h$-vector agrees with the usual $h$-vector for simplicial complexes, it is easy to see that if we fix the order complex homeomorphism type of a semi-Eulerian poset,  (\ref{toric formula}) spans all of the linear relations among the $\th_i$.

\section{inequalities} \label{inequalities}
   
  There are two very general inequalities for $h$-vectors of homology manifolds.  One is due to Schenzel (Theorem \ref{schen} below), and the following {\it rigidity inequality}  due to Kalai and, independently,  Gromov.
  
  \begin{thm} (Rigidity inequality) \cite[2.4.10]{Gr} \cite{Kal2} \label{rigidity}
    Suppose $\Delta$ is a   homology manifold without boundary and $d \ge 3.$ Then $h_0 \le h_1 \le h_2.$
  \end{thm}

The rigidity inequality has a strong implication when the fundamental group of $\Delta$ is nontrivial.  Suppose $\tilde{\Delta}$ is a $t$-sheeted covering of $\Delta.$  Then the triangulation of $\Delta$ lifts to a triangulation of $\tilde{\Delta}$ with $f_i(\tilde{\Delta}) = t f_i(\Delta)$  for $i \ge 0.$
 
 \begin{prop} \label{covering}
   If $\tilde{\Delta}$ is a $t$-sheeted covering of $\Delta, $ then
   \begin{itemize}
      \item
         $h_1(\tilde{\Delta}) =t \cdot h_1(\Delta) +d (t-1).$
      \item
         $h_2(\tilde{\Delta}) = t \cdot  h_2(\Delta) - (t-1) \binom{d}{2}.$
   \end{itemize}
\end{prop}

\begin{proof} This is a straight-forward application of  $f_i(\tilde{\Delta}) = t f_i(\Delta)$ and the definition of $h$-vectors in terms of $f$-vectors. 
\end{proof}

\begin{thm} \label{covering 2}
   Let $\Delta$ be a closed homology manifold.  If $\pi_1(\Delta)$ has a subgroup of index $t,$ then 
   $$\frac{t-1}{t} \binom{d+1}{2} \le h_2 - h_1 \le \binom{h_1}{2}.$$
   In particular, if $|\pi_1(\Delta)|$ is finite and greater than $\binom{d+1}{2},$ or,  if $\beta_1 > 0,$ then 
   $$\binom{d+1}{2} \le h_2 - h_1 \le \binom{h_1}{2}.$$
\end{thm}

\begin{proof}
The inequality $ h_2 - h_1 \le \binom{h_1}{2}$ holds for {\it any} pure  complex.  Let $\tilde{\Delta}$ be a $t$-sheeted covering space of $\Delta$ corresponding to a subgroup of $\pi_1(\Delta)$  of index $t.$  By the rigidity inequality, $ 0 \le h_2(\tilde{\Delta}) - h_1(\tilde{\Delta}).$ But, by the above proposition,

$$ 0 \le h_2(\tilde{\Delta}) - h_1(\tilde{\Delta}) = t \cdot h_2(\Delta) - (t-1) \binom{d}{2} - t \cdot  h_1(\Delta) - d(t-1).$$

If $\beta_1 \ge 0,$ then $\pi_1$ has subgroups of arbitrarily large index, so the second inequality follows from the first.

\end{proof}
         
The inequality involving only $h_1$ in Theorem \ref{covering 2} can be improved if $\Delta$ is a combinatorial manifold. A {\it combinatorial $(d-1)$-manifold} is a simplicial complex in which the link of every vertex is PL-homeomorphic to the boundary of the $(d-1)$-simplex.  

\begin{thm} \cite{BKu}
  Let $\Delta$ be a  combinatorial manifold.  If $\pi_1(\Delta)$ is not trivial, then $ d+1 \le h_1.$
\end{thm}

\begin{prob}
  Do there exists triangulated  manifolds with nontrivial fundamental group and $h_2 - h_1 < \binom{d+1}{2}$ or $h_1 < d+1?$  
   \end{prob}

One application of Theorem \ref{covering 2} is a proof that a family of triangulations of spherical bundles over $S^1$ given by K\"uhnel  have then minimum possible $f$-vector for homology manifolds without boundary and nonzero first Betti number. This family of complexes has the following properties.

\begin{thm} \cite{Ku2}
  For every  $d \ge 3$  there is a simplicial complex $M^d$ with the following properties.
  \begin{itemize}
      \item
         $M^d$ has $2d+1$ vertices.
      \item
         $M^d$ is $2$-neighborly, i.e. $f_1(M^d) = \binom{2d+1}{2}.$
      \item
         If $d$ is odd, then $M^d$ is homeomorphic to $S^1 \times S^{d-2}.$  If $d$ is even, then $M^d$ is homeomorphic to the nonorientable $S^{d-2}$-bundle over $S^1.$
     \item
         $M^d$ is vertex transitive with dihedral symmetry group.
      \item
         The link of every vertex of $M^d$ is a stacked sphere.
   \end{itemize}
\end{thm}

Our $M^d$ is called $M^{d-1}$ in \cite{Ku2}.  These triangulations were generalized by K\"uhnel and Lassmann \cite{KL}.  While we will consider  all of the K\"uhnel-Lassmann triangulations of  $S^1 \times S^{2m-1}$ in Section \ref{Constructions}, we refer the reader to \cite{KL} for  details on the others.

\begin{thm}  \label{min nonzero betti}
   If $\Delta$ is a homology manifold without boundary and nonzero first Betti number, then for all $i, f_i(\Delta) \ge f_i(M^d).$
\end{thm}

\begin{proof}
By Theorem \ref{covering 2}, $f_0(\Delta) \ge 2d+1 = f_0(M^d)$ and $f_1(\Delta) \ge \binom{2d+1}{2} = f_1(M^d).$   Define 

\begin{equation} \label{ss f def}
\tilde{f}_i(\Delta) = \displaystyle\sum^n_{j=1} f_i(\lk_\Delta v_j ).
\end{equation}

\noindent  Since $f_{i+1} = \tilde{f}_i / (i+2),$ it suffices to prove that 

\begin{equation} \label{ss f}
\tilde{f}_i(\Delta) \ge \tilde{f}_i(M^d).
\end{equation}  

Recall that $\phi_i(n,d)$ is the minimal number of $i$-faces in a $(d-1)$-dimensional  homology manifold without boundary which has $n$ vertices.  Define 
$$\Phi_i(N,n,d) =  \displaystyle\sum^n_{j=1} \phi_i (N_j,d),$$
where $N = N_1 + \dots + N_n$ is any composition of $N$ into $n$ nonzero parts.  The formula for $\phi,$ Equation (\ref{phi}), implies that this definition is independent of the choice of the $N_j.$

Now let $N_j$ be the number of vertices in the link of $v_j.$  So, $N_1 + \dots + N_n = \tilde{f}_0(\Delta) = 2 f_1(\Delta) \ge 2 f_1(M^d)= 2d(2d+1).$ Theorem \ref{Kal3} tells us that 
$$\tilde{f}_i(\Delta) \ge \displaystyle\sum^n_{j=1} \phi_i(N_j,d-1) = \Phi_i(2f_1(\Delta),f_0(\Delta),d-1).$$

   Theorem \ref{covering 2} says that $f_1(\Delta) \ge d f_0 (\Delta).$  As $\Phi_i(N,n,d)$ is monotonically increasing for fixed $n,d$ and $i, \tilde{f}_i(\Delta)$ will be minimized by the least value of $\Phi_i((d-1)n,n,d-1).$  However, for fixed $d \ge 4, 1 \le i \le d-1$ this function is strictly increasing as a  function of $n.$ Since $M^d$ minimizes $n$ and $\tilde{f}_i(M^d) = \Phi_i((d-1) n, n, d-1),$ where $n=2d+1,$ we are done.

\end{proof}

   Another way to use the rigidity  inequality is to sum it over the links of all the faces of a fixed dimension.  For this purpose we consider a generalization of the short simplicial $h$-vector introduced by Hersh and Novik in \cite{HN}.
   
   \begin{defn}
   $$\Th^{(m)}_i(\Delta) = \displaystyle\sum_{|\rho|=m} h_i( \lk \rho).$$
   \end{defn}
   
   The vector $(\Th^{(1)}_0, \dots, \Th^{(1)}_{d-1})$ was called the short simplicial $h$-vector in \cite{HN}.  
   
   \begin{prop} \cite{Sw7} 
      \begin{equation} \label{ss h}
  (m+1)\  \Th^{(m+1)}_{i-1} = i \Th^{(m)}_i + (d-m-i+1) \Th^{(m)}_{i-1}.
  \end{equation}
   \end{prop} 
   
   As long as  all of the links in question are homology manifolds of dimension at least three, the rigidity inequality implies $\Th^{(m)}_0 \le \Th^{(m)}_1 \le \Th^{(m)}_2.$  Here are two examples of this principle.

   \begin{thm}  \label{min h}
    Suppose the link of every vertex of $\Delta$ is a $(d-2)$-dimensional  homology manifold without boundary.  Then 
    \begin{equation} \label{walkupgen}
    (d-1) h_1 \le 3 h_3 + (d-4) h_2.
    \end{equation}
      Furthermore, when  $d \ge 5,$ equality occurs if and only if $\Delta \in \mathcal{H}^{d-1}.$  In this case, the  $h$-vector of $\Delta$ is determined by $h_1$ and $h_2.$
  \end{thm}
  
  \begin{proof}
  From the previous proposition, $\Th^{(1)}_1 = 2 h_2 + (d-1) h_1$ and $\Th^{(1)}_2 = 3 h_3 + (d-2) h_2.$  The rigidity inequality applied to these two equations gives (\ref{walkupgen}).  In addition,  equality occurs if and only if  for each vertex $v$ of $\Delta, h_1(\lk v) = h_2 (\lk v).$  By Corollary \ref{Kal2} each such link must be  a stacked sphere   and thus $h_1(\lk v) = h_2(\lk v) = \cdots = h_{d-2}(\lk v).$  Hence $\Th^{(1)}_1 = \Th^{(1)}_2 = \cdots = \Th^{(1)}_{d-2}.$  Since $h_1$ and $h_2$ determine $\Th^{(1)}_1$ they determine all of the $\Th^{(1)}_i.$ It is not difficult to see that (\ref{ss h}) insures that this determines the entire $h$-vector.  Finally, the link of every vertex  of $\Delta$ is a stacked sphere if and only if $\Delta \in \mathcal{H}^{d-1}.$
  
  \end{proof}
  
  The above theorem is optimal in the following sense.  When $d=4$ (\ref{walkupgen}) reduces to $h_1 \le h_3.$ Both $h_1 = h_3$ (any homology manifold without boundary) and $h_1 < h_3$ (for instance, the suspension of $\R P^2$) can occur.  For $d>4,$ any triangulation in $\mathcal{H}^{d-1}$ satisfies equality, including, for instance, the K\"{u}hnel-Lassmann triangulations of $S^1 \times S^{2m+1}$ and the nonorientable $S^{2m}$ bundle over $S^1$ \cite{KL},  stacked spheres, and connected sums along facets of any of these spaces.

  When $d=5$ (\ref{walkupgen}) becomes
  $$\begin{array}{lcl}
  4 h_1 & \le & 3h_3 + h_2\\
  \Rightarrow 4(h_1 - h_2) & \le & 3(h_3 - h_2).
  \end{array}$$
Any three manifold without boundary has Euler characteristic zero, so $\Delta$ is semi-Eulerian and we can substitute $10(\chi(\Delta) - 2)$ for $h_3 - h_2.$ Hence,
 \begin{equation} \label{walkup}
 h_2 - h_1 \ge -\frac{15}{2} (\chi(\Delta)-2).
 \end{equation}

 This formula first appears in \cite{Wal}, as does the characterization of equality.   As Walkup's proof is logically equivalent to  the one above, (\ref{walkupgen}) can be viewed as a higher dimensional analog of (\ref{walkup}).  For another example, we consider $m=2.$

 \begin{thm}  
   If the link of every edge is a homology manifold without boundary and $d \ge 5,$  then 
   \begin{equation} \label{walkupgen2}
   12 h_4 +6(d-4) h_3 + (d-2)(d-7) h_2 - (d-1)(d-2) h_1 \ge 0.
   \end{equation}
 \end{thm}
 
 \begin{proof}
 From (\ref{ss h}), $2 h^{(2)}_1 = 2 h^{(1)}_2 + (d-2)h^{(1)}_1$ and $2h^{(2)}_2 = 3 h^{(1)}_3 + (d-3) h^{(1)}_2.$  Applying (\ref{ss h}) again to the right hand side of these equations,
 $$\begin{array}{lcl}
   2 h^{(2)}_1&  = & 2[3 h_3 + (d-2) h_2]+(d-2)[2 h_2 + (d-1) h_1] \\
   2h^{(2)}_2 & = & 3[4 h_4 + (d-3) h_3] + (d-3) [ 3h_3 + (d-2) h_2].
\end{array} $$                 
 
 The rigidity inequality implies $2h^{(2)}_2 \ge 2 h^{(2)}_1.$ 
  \end{proof}
  
  \begin{cor}
  If $\Delta$ is a $6$-dimensional homology manifold without boundary, then
  $$\chi(\Delta) \le 2 + \frac{1}{14} (h_3 - h_1).$$
  Furthermore,  $\chi(\Delta) = 2 + \frac{1}{14} (h_3 - h_1),$ if and only if $\Delta \in \mathcal{H}^{d-1}.$
    \end{cor}
  
  \begin{proof}
  Setting $d=7,$ (\ref{walkupgen2}) is equivalent to
  $$12(h_4 - h_3) \ge 30(h_1 - h_3).$$
  By the generalized Dehn-Sommerville equations $h_4 - h_3 = -35(\chi(\Delta) -2).$ So,
  $$\chi(\Delta) - 2 \le \frac{30}{12 \cdot 35}(h_3 - h_1).$$
 
 Now suppose $\chi(\Delta) = 2 + \frac{1}{14} (h_3 - h_1)$.  Let $\Delta_v$ be the link of a vertex of $\Delta.$  The proof of (\ref{walkupgen2}) shows that for every vertex $w \in \Delta_v,$ $h_1(\lk_{\Delta_v} w) = h_2(\lk_{\Delta_v} w).$ Indeed, equality holds if and only if this is true.   Since $\Delta$ is a homology manifold, $H_1(\Delta_v; \Q)=0.$  Kalai's first proof of  \cite[Theorem 1.1]{Kal2} shows that $\Delta_v$ is  a stacked sphere and thus $\Delta \in \mathcal{H}^{d-1}.$  Conversely, suppose $\Delta \in \mathcal{H}^{d-1}.$ Since the link of every vertex of a stacked sphere is a stacked sphere, $h^{(2)}_1 = h^{(2)}_2$ and equality in (\ref{walkupgen2}) follows.
  \end{proof}
  
  \begin{cor}
    If $\Delta$ is a $6$-dimensional Eulerian homology manifold,  then the $h$-vector of $\Delta$ is positive.
  \end{cor}

   For further estimates we turn to the face ring.
  By introducing a linear system of parameters, the face ring can be a powerful tool in understanding the combinatorics of $\Delta$.     Let $\Theta=\{\theta_1,\dots,\theta_d\}$ be a set of one forms in $R=k[x_1, \dots, x_n].$  For each $i,$ write $\theta_i = \theta_{i,1} x_1 + \dots + \theta_{i,n} x_n$ and for each facet $\sigma \in \Delta$ let $T_\sigma$ be the $d \times d$ matrix whose entries are $\{\theta_{i,j}\}_{v_j \in \sigma}.$ We say $\Theta$ is a {\bf linear system of parameters} (l.s.o.p.) for $k[\Delta]$ if $T_\sigma$ has rank $d$ for every facet $\sigma$ of $\Delta.$         
 
 \begin{thm}(Schenzel's formula) \cite{Sche} \label{schen}
 
 Let $\Theta$ be a l.s.o.p. for $k[\Delta],$  and let $h^\prime_i =\dim_k  (k[\Delta]/\langle \Theta \rangle)_i.$ If $\Delta$ is a $k$-homology manifold (with or without boundary), then
 
 \begin{equation} \label{schenzel}
  h^\prime_i = h_i + \binom{d}{i} \displaystyle\sum^{i-1}_{j=2} (-1)^{i-j-1} \beta_{j-1}, \end{equation}
 where the $\beta_j$ are the $k$-Betti numbers of $\Delta.$
 \end{thm}

Schenzel's proof of the above formula applies to the much more general class of connected  Buchsbaum complexes.
As an application of (\ref{schenzel}), we note that if $\Delta$ is a homology manifold without boundary, then $h^\prime_d(\Delta) = 1$ if $\Delta$ is orientable and $h^\prime_d(\Delta) = 0 $ if $\Delta$ is not orientable.

Schenzel's formula frequently allows us to move back and forth between the commutative algebra of $k[\Delta]/\langle \Theta \rangle,$ and the combinatorics of $\Delta.$  As an example, the rigidity inequality has an interpretation in $k[\Delta]/\langle \Theta \rangle$ due to Lee. 

 \begin{thm} \cite{Le2} \label{injectivity}
    Let $\Delta$ be a  homology manifold without boundary and let $k$ be a field of characteristic zero.  For generic pairs $(\omega,\Theta), \omega$ a one form in $R$ and $\Theta$ a l.s.o.p. for $k[\Delta],$ multiplication
    $$\omega: (k[\Delta]/\langle \Theta \rangle)_1 \to (k[\Delta]/\langle \Theta \rangle)_2$$
    is an injection.
  \end{thm}

 In view of Theorem \ref{schen}, the study of  $f$-vectors, $h$-vectors and $h^\prime$-vectors, where we define $h^\prime_i$ by (\ref{schenzel}), are all equivalent for homology manifolds.  The value of the $h^\prime$-vectors is that for homology manifolds $(h^\prime_0, h^\prime_1, \dots, h^\prime_d)$ is the Hilbert function of $k[\Delta]/\langle \Theta \rangle,$ and Hilbert functions of homogeneous quotients   of polynomial rings were characterized by Macaulay.
 
 Given $a$ and $i$ positive integers there is a unique way to  write
$$a = \binom{a_i}{i} + \binom{a_{i-1}}{i-1} + \dots + \binom{a_j}{j},$$
with $a_i > a_{i-1} > \dots > a_j \ge j \ge 1.$

Define 
$$a^{<i>} = \binom{a_i+1}{i+1} + \binom{a_{i-1}+1}{i} + \dots + \binom{a_j+1}{j+1}.$$

\begin{thm} \cite[II.2.2]{St} \label{Macaulay}
  Let $(h_0, \dots, h_d)$ be a sequence of nonnegative integers. Then the following are equivalent.
  
  \begin{itemize}
     \item
       $(h_0, \dots, h_d)$ is the Hilbert function of a homogeneous quotient of a polynomial ring.
     
     \item
       $h_0=1$ and $h_{i+1} \le h^{<i>}_i$ for all $1 \le i \le d-1.$
   \end{itemize}
 \end{thm}

\noindent  Any sequence $(h_0, \dots, h_d)$ which satisfies the above conditions is called an {\it M-vector}.

 In combination with our previous results and the following theorem, Macaulay's formula leads to restrictions on $n$ and  $h_2-h_1$ for $2m$-dimensional homology manifolds without boundary. 

\begin{thm} \cite{No} \label{kernel}
  Let $\Delta$ be a homology manifold, $\Theta$ a l.s.o.p. for $k[\Delta],$ and $\omega \in R_1$  (one forms in $R$). Then the kernel of multiplication $\omega: (k[\Delta]/\langle \Theta \rangle)_i \to (k[\Delta]/\langle \Theta \rangle)_{i+1}$ has dimension greater than or equal to $\binom{d-1}{i} \beta_{i-1}.$
\end{thm}

\noindent Kalai has conjectured that the correct lower bound for homology manifolds without boundary is $\binom{d}{i} \beta_{i-1}$ \cite[Conjecture 36]{Kal3}.

 For connected $2m$-dimensional  homology manifolds without boundary $h^\prime_{m+1} - h^\prime_m$ does not depend on the triangulation. Define 
$$\begin{array}{rcl}
G(\Delta) & = & (-1)^m \binom{2m+1}{m} [(\beta_1 - \beta_{2m-1})-(\beta_2 - \beta_{2m-2}) + \dots \\
& &\dots  + (-1)^{m-1}(\beta_{m-2} - \beta_{m+2}) + (-1)^m(\beta_m - \beta_{m+1})]. \end{array}$$

\noindent  By the generalized Dehn-Sommerville equations and Schenzel's formula, $h^\prime_{m+1} - h^\prime_{m} = G(\Delta)$ \cite{No}.  If $\Delta$ is orientable, then $G(\Delta)$ reduces to $\binom{2m+1}{m}(\beta_m - \beta_{m-1}).$

\begin{thm} \label{even euler inequality}
Let $\Delta$ be a $2m$-dimensional  connected homology manifold without boundary and suppose $G(\Delta) > 0,$ where $G(\Delta)$ is computed using rational coefficients.  Write $h_2 - h_1 = \binom{a}{2} + \binom{b}{1},$ with $a > b.$  Then
\begin{enumerate}
  \item[(a)] \label{kn1}
     $G(\Delta) + \binom{2m}{m} \beta_{m-1} \le \binom{n-m-2}{m}.$
  \item[(b)] \label{kn2}
     $G(\Delta) + \binom{2m}{m} \beta_{m-1}\le \binom{a+m-1}{m+1} + \binom{b+m-1}{m} .$
\end{enumerate}
\end{thm}

\begin{proof}
Let $\Theta$ be a l.s.o.p. for $\Q[\Delta]$ and  $\omega \in \Q[\Delta]_1$ satisfy Theorem \ref{injectivity}.  Define $g^\prime_i = \dim_\Q (\Q[\Delta]/\langle \Theta,\omega \rangle)_i.$  Then $g^\prime_{m+1} \ge h^\prime_{m+1} - h^\prime_m + \binom{2m}{m} \beta_{m-1},$ with equality if and only if the dimension of the kernel of multiplication $\omega:(\Q[\Delta])/\langle \Theta \rangle)_m \to  (\Q[\Delta])/\langle \Theta \rangle)_{m+1}$ is $\binom{2m}{m} \beta_{m-1}$.  The choice of $\Theta$ and $\omega$ imply that $g^\prime_2 = h^\prime_2 - h^\prime_1$ and by Schenzel's formula this is $h_2 - h_1.$   The inequalities now follow from Macaulay's arithmetic criterion for Hilbert functions.

\end{proof}

If $\beta_{m-1}=0$ and $\Delta$ is orientable, then the left hand side of both inequalities reduces to $\binom{2m+1}{2m} \beta_m.$ If Kalai's conjecture concerning the lower bound for the dimension of the kernel of multiplication by a one form is correct, then again the left hand side of both inequalities reduce to $\binom{2m+1}{2m} \beta_m$ whenever $\Delta$ is orientable.  

While the second inequality is new and a key ingredient to the complete characterization of $f$-vectors in the next section, the first inequality is neither new nor best.  See \cite[Theorem 5.7]{No} for a related stronger inequality and a discussion.

When $\Delta$ is a homology sphere or ball $h_i = h^\prime_i.$  In the special case of $\Delta$ equal to  the boundary of a simplicial polytope even more can be said.   In \cite{Mc3} P. McMullen conjectured the following characterization of  $h$-vectors of simplicial polytopes.  
\begin{conj}\cite{Mc3}
  A sequence $(h_0,h_1,\dots,h_d)$ is the $h$-vector of the boundary of a simplicial $d$-polytope if and only if 
    \begin{enumerate}
       \item[(a)]
          $h_0=1.$
        \item[(b)]
          $h_0 \le h_1 \le \dots \le h_{\lfloor d/2 \rfloor}.$
        \item[(c)]
          $(h_0, h_1 - h_0, \dots, h_{\lfloor d/2 \rfloor} - h_{\lfloor d/2 \rfloor - 1})$ is an M-vector.
     \end{enumerate}
 \end{conj}

The sequence $(g_0,\dots,g_{\lfloor d/2 \rfloor}) = (h_0, h_1 - h_0, \dots, h_{\lfloor d/2 \rfloor} - h_{\lfloor d/2 \rfloor - 1})$ is usually called the {\bf g-vector} of $\Delta.$  The correctness of McMullen's conjecture was proved in two separate papers. In their 1981 paper \cite{BLe} Billera and Lee showed how to  construct a simplicial $d$-polytope with a given $h$-vector whenever it satisfied McMullen's conditions.  Stanley's proof of the necessity of McMullen's conditions used a hard Lefschetz theorem for toric varieties  associated to rational polytopes \cite{St2}.  The main point is that, generically, $k[\Delta]/\langle \Theta \rangle$ has Lefschetz elements.

\begin{defn}
  Let $\Delta$ be a homology sphere.  A {\bf Lefschetz} element for $k[\Delta]/\langle \Theta \rangle$ is a one form $\omega \in R$ such that for all $i \le \lfloor d/2 \rfloor$ multiplication 
  $$\omega^{d-2i}: (k[\Delta]/\langle \Theta \rangle)_i \to (k[\Delta]/\langle \Theta \rangle)_{d-i},$$
  is an isomorphism.
 \end{defn}
 
 Suppose that $\omega$ is a Lefschetz element for $k[\Delta]/\langle \Theta \rangle.$  Then for $i \le \lfloor d/2 \rfloor$ multiplication by omega, $\omega: (k[\Delta]/\langle \Theta \rangle)_{i-1} \to (k[\Delta]/\langle \Theta \rangle)_i$ must be an injection.  Hence, for such $i, h_{i-1} \le h_i.$  In addition, we can see that $g_i = \dim_k (k[\Delta]/\langle \Theta, \omega \rangle)_i, $ so the $g$-vector of $\Delta$ must be an M-vector. Hence,  if $L(\Delta) = \{(\omega, \Theta): \omega \mbox{ is a Lefschetz element for } k[\Delta]/\langle \Theta \rangle \}$ is nonempty, then $\Delta$ satisfies McMullen's conditions.
 
 \begin{thm} \cite{St2}, \cite{Mc2} \label{g-thm}
   If $\Delta$ is the boundary of a simplicial polytope, then $L(\Delta)$ is nonempty.
\end{thm}
 
 Perhaps the most important problem involving $f$-vectors is whether or not McMullen's conditions extend to simplicial spheres or even homology spheres.  This question is sometimes referred to as the $g$-conjecture.  As the above discussion shows, the following {\it algebraic} $g$-conjecture would imply the $g$-conjecture. 
  
 \begin{conj}  \label{g-conj}
   If $\Delta$ is a homology sphere, then $L(\Delta) \neq \emptyset.$
 \end{conj}
 
 \noindent For a related, even stronger conjecture, see \cite{Sw7}.
 
 \begin{defn}
   A {\bf Lefschetz sphere} is a homology sphere $\Delta$ such that $L(\Delta) \neq \emptyset.$ A {\bf Lefschetz ball} is a homology ball $\Delta$ which is a full dimensional subcomplex of a Lefschetz sphere.
 \end{defn}
 
   Unlike homology spheres,  for an arbitrary homology manifold $\Delta$ we do not expect multiplication by a one form in $k[\Delta]/\langle \Theta \rangle$  to be an injection in degrees larger than two.  Indeed, by Theorem \ref{kernel} the dimension of the kernel of multiplication $\omega:  (k[\Delta]/\langle \Theta \rangle)_i \to (k[\Delta]/\langle \Theta \rangle)_{i+1}$ is at least $\binom{d-1}{i} \beta_{i-1}$. However, as we will now show, if ``enough'' links of the vertices of $\Delta$ are Lefschetz spheres or balls, then multiplication by a generic one form is a surjection in the higher degrees.

   Let $\rho=\{v_{m_1}, \dots, v_{m_{|\rho|}}\}$ be an ordered face of $\Delta,$ and let $\sigma=\{v_{m_1}, \dots,v_{m_{|\rho|}}, v_{m_{|\rho|+1}}, 
\dots,  v_{m_d}\}$ be a facet containing $\rho.$  Suppose $\Theta$ is a l.s.o.p. for $k[\Delta]$ with corresponding matrix $T = \theta_{i,j}.$  Then there is a unique set of one forms $\Theta^\prime = \{\theta^\prime_1, \dots, \theta^\prime_d\}$ such that $\langle \Theta \rangle = \langle \Theta^\prime \rangle$ and for any $i, 1 \le i \le d,$

$$\theta^\prime_{i,j} = \begin{cases} 1, j=m_i , \\
                                                           0, j=m_l, 0 \le l \le d, l \neq i.  \end{cases}$$
 Indeed, $T^\prime$  corresponds to the reduced row echelon form of $T$ with pivot columns $\{m_1, \dots, m_d \}.$  For future reference, we note that the $\theta^\prime_{i,j}$ are rational functions of the $\theta_{i,j}.$  
 
 In order to use $\Theta^\prime$ as a l.s.o.p. for $k[\lk \rho],$ let $R^\rho$ be the polynomial ring over $k$ with variables $\{x_i\}_{v_i \notin \rho}.$  For each $\theta^\prime_i,$ let $\theta^\rho_i$ be the one form in $R^\rho$ obtained from $\theta^\prime_i$ by removing all the variables corresponding to vertices in $\rho.$  Equivalently, $\theta^\rho_i$ is the image of $\theta^\prime_i$ under the natural surjection from $R$ to $R^\rho.$ It is now easy to check that $\Theta^\rho = \{\theta^\rho_{|\rho|+1},\dots,\theta^\rho_d\}$ is a l.s.o.p. for $k[\lk \rho].$  While these definitions depend on the choice of facet $\sigma \supseteq \rho$, we will suppress this dependence as it will not matter.
 
 We intend to analyze ideals of the form $\langle x_\rho \rangle \subseteq k[\Delta]/\langle \Theta \rangle$, where $x_\rho = x_{m_1} \cdots x_{m_{|\rho|}},$ by using $k[\lk \rho]/ \langle \Theta^\rho \rangle.$  This requires us to give $k[\lk \rho]/\langle \Theta^\rho \rangle$ an $R$-module structure.  It is sufficient to describe $x_{m_i} \cdot q(\xx)$ for $q(\xx) \in k[\lk \rho]/\langle \Theta^\rho \rangle.$ For each $v_{m_i} \in \rho$ the construction of $\Theta^\prime$   forces $\theta^\prime_{m_i}$ to be of the form
 
 $$\theta^\prime_{m_i} = x_{m_i} + \displaystyle\sum_{v_j \notin \sigma} \theta^\prime_{m_i,j} x_j.$$
 
 \noindent So we define 
 \begin{equation} \label{R mod}
 x_{m_i} \cdot q(\xx) = (- \displaystyle\sum_{v_j \notin \sigma} \theta^\prime_{m_i,j} x_j) \cdot q(\xx).
 \end{equation}
 
 As $\theta^\prime_i$ is in $\Theta,$ this definition insures that multiplication by $x_\rho$ is an $R$-module homomorphism of degree $|\rho|$ from $k[\lk \rho]/\langle \Theta^\rho \rangle$ to the ideal $\langle x_\rho \rangle$ in $k[\Delta]/\langle \Theta \rangle.$  Indeed, if $R^\rho$ is given an $R$-module structure in the same way as $k[\lk \rho],$ then there is a commutative diagram of $R$-modules.
 
 $$\begin{array}{ccc}
  R^\rho/\langle \Theta^\rho \rangle & \stackrel{\cdot x_\rho}{\longrightarrow} & <x_\rho>(R/ \langle \Theta^\prime \rangle)\\
   \downarrow & & \downarrow \\
   k[\lk \rho]/ \langle \Theta^\rho \rangle& \stackrel{\cdot x_\rho}{\longrightarrow} & <x_\rho> (k[\Delta]/\langle \Theta \rangle)
   \end{array}$$
 
 \begin{prop}  \label{link to ideal}
   Let $\Delta$ be a homology manifold.  Then the multiplication map
   $$x_{\rho}: k[\lk \rho]/\langle \Theta^\rho \rangle \to \langle x_\rho \rangle (k[\Delta]/\langle \Theta \rangle)$$ is a surjective graded $R$-module homomorphism of  degree $|\rho|.$  If $\Delta$ is a   homology manifold without boundary, then $x_\rho$ is an isomorphism unless $i=d$ and $\Delta$ is not orientable. 
  \end{prop}
  
  \begin{proof}
 Evidently, the map has degree equal to the cardinality of $\rho.$ To see that the map is surjective, let $x_\rho \cdot q(\xx) \in \langle x_\rho \rangle.$  Replace each occurrence of a variable $x_{m_i}$ in $q(\xx)$ using (\ref{R mod}).  This leaves a polynomial which is clearly in the image of multiplication by $x_\rho.$  
  
  In order to show that multiplication is an isomorphism when $\Delta$ is a homology manifold without boundary and either $\Delta$ is orientable or $i \neq d,$ it is sufficient to show that the dimensions over $k$ agree.  Since $\Delta$ has no boundary, the link of $\rho$ is a homology sphere, so $\dim_k (k[\lk \rho]/\langle \Theta^\rho \rangle)_i$ is $h_i(\lk \rho).$  To compute $\dim_k \langle x_\rho \rangle_{i+|\rho|},$ consider the exact sequence,
  
  $$0 \to \langle x_\rho \rangle \to k[\Delta]/\langle \Theta \rangle \to k[\Delta - \rho]/\langle \Theta \rangle \to 0,$$
  where $\Delta-\rho$ is $\Delta$ with $\rho$ and any incident faces removed.  Since $\Delta - \rho$ is a homology manifold with boundary, Schenzel's formula says
  $$\dim_k \ \langle x_\rho \rangle_{i+|\rho|} = \dim_k (k[\Delta]/\Theta)_{i+|\rho|} - \dim_k (k[\Delta-\rho]/\Theta)_{i+|\rho|} $$
  $$= h^\prime_{i+|\rho|}(\Delta) - h^\prime_{i+|\rho|}(\Delta-\rho).$$
  
  The Mayer-Vietoris sequence for $\Delta = (\Delta - \rho) \cup \overline{\st} \rho$  shows that if $\Delta$ is orientable or $j \neq d-2,$ $\beta_j(\Delta) = \beta_j(\Delta - \rho). $ Hence, $\dim_k \ \langle x_\rho \rangle_{i+|\rho|} = h_{i+|\rho|}(\Delta) - h_{i+|\rho|}(\Delta -\rho).$  This difference is the coefficient of $t^{i+|\rho|}$ in 
  
  $$\displaystyle\sum_{\stackrel{\phi \in \overline{\st}\rho}{\phi \notin \partial \overline{\st}\rho}} (t-1)^{d-|\phi|}.$$
  
  This is known to be $h_{d-i-|\rho|}(\overline{\st}\rho)$ \cite[Lemma 2.3]{St3}.  Since the $h$-vector of a cone is the $h$-vector of the original space, $h_{d-i-|\rho|}(\overline{\st}\rho) = h_{d-i-|\rho|}(\lk \rho).$  As the dimension of the link of $\rho$ is $d-1-|\rho|,$ the generalized Dehn-Sommerville equations show that this is $h_i(\lk \rho).$
  \end{proof} 
 
   Define
$L^i_s(\Delta)$ to be the set of all pairs $(\omega, \Theta)$ such that $\omega \in R_1, \Theta$ is a l.s.o.p. for $\Delta,$ and multiplication 
     $$ \omega:(k[\Delta]/\langle \Theta \rangle)_i \to (k[\Delta]/\langle \Theta \rangle)_{i+1}$$
     is a surjection.  If $\Delta$ is a homology sphere, then $L(\Delta) \subseteq L^{\lfloor d/2 \rfloor}_s(\Delta).$ 
          
    \begin{prop} \label{L^i}
      If $\Delta$ is a homology manifold and $L^i_s(\Delta) \neq \emptyset,$ then for all $j, i \le j \le d-1,$
      $$h^\prime_j \ge h^\prime_{j+1} + \binom{d-1}{j} \beta_{j-1}.$$
   \end{prop}
   
   \begin{proof}  Let $(\omega,\Theta) \in L^i_s(\Delta).$  Then for any $j, i \le j \le d-1,$ multiplication $\omega: (k[\Delta]/\langle \Theta \rangle)_j \to (k[\Delta]/\langle \Theta \rangle)_{j+1}$ is a surjection with a kernel whose dimension is at least $\binom{d-1}{j} \beta_{j-1}.$
   
   \end{proof}

   \begin{thm}  \label{g-thm to manifold}
    Suppose $\Delta$ is a $k$-homology manifold  with $k$ an infinite field.  If for at least $n-d$ of the vertices $v$ of $\Delta, L^i_s(\lk v) \neq \emptyset,$ then $L^{i+1}_s(\Delta) \neq \emptyset.$ 
   \end{thm}
   
  \begin{proof}
    Let $V_s=\{v_j\}^{n-d}_{j=1}$ be vertices of $\Delta$ such that for every $j,$ $L^i_s(\lk v_j) \neq \emptyset.$ For each such vertex $v$, consider the set of pairs 
    $$L^i_s(v) = \{(\omega, \Theta):\Theta \mbox{ is a l.s.o.p. for } k[\Delta], \mbox{and } (\omega, \Theta^{\{v\}}) \in L^i_s(\lk v)\}.$$ 
    
    \noindent Since $\Theta \to \Theta^{\{v\}}$ is a rational map and $L^i_s(\lk v)$ is a nonempty Zariski open set \cite{Sw7}, $L^i_s(v)$ is a nonempty Zariski open set.  We call $\Theta$ {\it generic} if every $d \times d$ minor of the associated matrix $T$ is nonsingular.   To finish the proof we show that 
    
    $$L = \{(\omega,\Theta): \Theta \mbox{ is generic}\} \cap \displaystyle\bigcap^{n-d}_{j=1} L^i_s(v)$$
    is a nonempty subset of $L^{i+1}_s(\Delta).$
  
Since each of the intersecting subsets in $L$ is a nonempty open Zariski subset of $k^{(d+1)n}, L$ is nonempty.  So let $(\omega, \Theta) \in L.$ In order to see that multiplication $\omega:(k[\Delta]/\langle \Theta \rangle)_{i+1} \to  (k[\Delta]/\langle \Theta \rangle)_{i+2}$ is surjective it is sufficient to show that every monomial in $(k[\Delta]/\langle \Theta \rangle)_{i+2}$ is in the image.  We consider two cases.

Case 1: The monomial can be written in the form $x_j \cdot x^\alpha,$ where $v_j \in V_s.$ Using the fact that $x^\alpha$ is in the image of multiplication by $\omega$ in $k[\lk v_j]/\langle \Theta^{\{v_j\}} \rangle,$ and Proposition  \ref{link to ideal}, we  see that this monomial is in the image of multiplication by $\omega$ in $\langle x_j \rangle \subseteq k[\Delta]/\langle \Theta \rangle.$
   
Case 2:  All of the variables in the monomial correspond to vertices not in $V_s.$  Write the monomial $x_l \cdot x^\alpha.$  Since $\Theta$ is generic, it contains an element $\theta$ of the form
$$\theta = x_l + \displaystyle\sum_{j, v_j \in V_s} \theta_j x_j.$$
This implies that the monomial is equivalent to a sum of monomials from Case 1 and hence is in the image of multiplication by $\omega.$
  \end{proof}

The above results suggest the following conjectures for homology manifolds (with or without boundary).

\begin{conj}(Manifold algebraic $g$-conjecture)

If $\Delta$ is a homology manifold, then $L^i_s(\Delta) \neq \emptyset$ for $i \ge \lceil d/2 \rceil.$
\end{conj}

\begin{conj}(Manifold $g$-conjecture)

If $\Delta$ is a homology manifold, then $h^\prime_i \ge h^\prime_{i+1} + \binom{d-1}{i} \beta_{i-1}$ for $i \ge \lceil d/2 \rceil.$
\end{conj}

If Kalai's conjecture concerning the kernel of multiplication by a one form is correct (see discussion  following Theorem \ref{kernel}), then the manifold $g$-conjecture would be $h^\prime_i \ge h^\prime_{i+1} + \binom{d}{i} \beta_{i-1}.$ Figure \ref{implications} shows the interrelationships among the various $g$-conjectures.  The dotted arrows indicate partial implications.  The manifold algebraic $g$-conjecture does not imply the existence of Lefschetz elements for homomlogy spheres.  However, the surjective maps promised by the manifold algebraic $g$-conjecture, combined with the Gorenstein property of  face rings of homology spheres, is enough to establish the injective maps needed to establish the $g$-conjecture.  The manifold $g$-conjecture would insure that $g$-vectors of homology spheres are nonnegative, but would not show that they are M-vectors.

\begin{figure} 
 \scalebox{0.50}[0.50]{\includegraphics{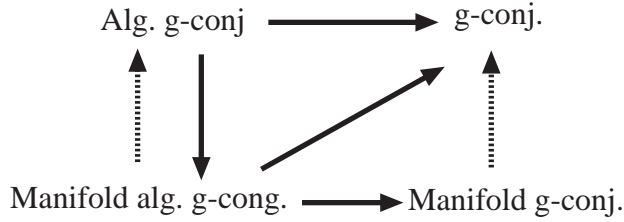}}
  \caption{Various $g$-conjectures} \label{implications}
\end{figure}

%Suppose $L(\Delta)$ is nonempty for all  homology spheres (Conjecture \ref{g-conj}).  Let $m = \lceil d/2 \rceil$. Then $L^m_s$ is nonempty for homology balls and spheres.  For homology balls this follows from \cite[Lemma 2.2]{St3}.  The above result  would then imply  that $L^m_s$ is not empty for all homology manifolds (with or without boundary). Conversely,  $L^m_s \neq \emptyset$ for all homology manifolds does not imply $L(\Delta)$ is not empty for all homology spheres.  In fact, it does not even imply that $L^{m-1}_s$ is  nonempty for homology spheres.   However, using the Gorenstein property of face rings of homology spheres (see the argument in the proof of \cite[Theorem 3.9]{Sw7}), it is possible to prove that if $\Delta$ is a homology sphere and $L^m_s(\Delta) \neq \emptyset,$ then the $h$-vector of $\Delta$ satisfies McMullen's conditions.  

Even without an affirmation of the algebraic $g$-conjecture, Theorem \ref{g-thm to manifold} can be used  effectively to limit the possible $h$-vectors of homology manifolds.   
  
  \begin{cor}
    If $k$ has characteristic zero and $\Delta$ is a $k$-homology manifold, then $L^{d-2}_s(\Delta) \neq \emptyset.$
  \end{cor}
  
  \begin{proof}
     Every two dimensional homology sphere $\Delta^\prime$ is the boundary of a simplicial 3-polytope, hence by Theorem  \ref{g-thm} $L(\Delta^\prime) \neq \emptyset.$ By \cite{St3} $L^1_s$ is nonempty for two dimensional homology balls.   Now apply induction and Theorem \ref{g-thm to manifold}.
  \end{proof}

  \begin{thm}  \label{h by betti}
    If $\Delta$ is a homology manifold and $\beta_i$ are the rational Betti numbers of $\Delta$, then 
    \begin{equation} \label{2 to 1}
    h^\prime_{d-2} \ge h^\prime_{d-1} + (d-1) \beta_{d-3} 
    \end{equation}
    If, in addition,  $\Delta$ is  closed, then
    \begin{equation} \label{orientable 2 to 1}
    h_2 \ge h_1 + \binom{d+1}{2} \beta_1 - \binom{d-1}{2} \beta_2.
    \end{equation}
    Furthermore, if $\Delta$ is closed, $d \ge5, \beta_2 = 0$ and $ h_2 =h_1 + \binom{d+1}{2} \beta_1,$ then $\Delta \in \mathcal{H}^{d-1}.$
    \end{thm}
    
    \begin{proof}  The first formula is an immediate consequence of Proposition \ref{L^i} and the above corollary.  So assume that $\Delta$ is closed and orientable.  By Schenzel's formula $h_1 = h^\prime_1$ and $h_2 = h^\prime_2.$  On the other hand, by \cite{No}, $h^\prime_{d-2} = h^\prime_2 + \binom{d}{2} (\beta_2 - \beta_1)$ and $h^\prime_{d-1} = h^\prime_1 + d \beta_1.$  Combining this with  (\ref{2 to 1}) and Poincar\'{e} duality gives (\ref{orientable 2 to 1}).
    
    Now suppose $\beta_2 = 0,  h_2 = h_1 + \binom{d+1}{2} \beta_1, d \ge 5$ and $\Delta$ is closed.   These conditions imply that $h^\prime_{d-2} = h^\prime_{d-1}.$ This means that for  $(\omega,\Theta) \in L^{d-2}_s(\Delta)$   multiplication $\omega: (k[\Delta]/\langle \Theta \rangle)_{d-2} \to (k[\Delta]/\langle \Theta \rangle)_{d-1}$ is a bijection. So it must be an injection when considered as a map $\omega: \langle x_i \rangle_{d-2} \to \langle x_i \rangle_{d-1}.$    By Proposition \ref{link to ideal}, multiplication $\omega:(k[\lk v_i]/\langle \Theta^{x_i} \rangle)_{d-2} \to (k[\lk v_i]/\langle \Theta^{x_i} \rangle)_{d-1}$ must be an injection.   Hence, for any vertex $v_i$ of $\Delta,$  $h_{d-2}(\lk v_i) \le h_{d-1}(\lk v_i).$  Each such link is a homology sphere, so the generalized Dehn-Sommerville equations imply $h_2 \le h_1$ in each vertex link. The rigidity inequality and Corollary \ref{Kal2} imply that every vertex link is a stacked sphere.  Finally, Theorem \ref{walk1} says that $\Delta \in \mathcal{H}^{d-1}.$
    
      \end{proof}

   Kalai conjectured in \cite{Kal2} that for  homology manifolds without boundary
\begin{equation} \label{Kalai conj}
h_2 - h_1 \ge \binom{d+1}{2} \beta_1.
\end{equation}   The above theorem verifies this conjecture when $\beta_2 = 0$ and $\Delta$ is closed. Theorem \ref{covering 2} confirms (\ref{Kalai conj}) when $\beta_1 = 1.$ If Kalai's conjecture concerning the lower bound for the dimension of the kernel of multiplication by a one form is correct (see comment following Theorem \ref{kernel}), the suitably altered statement of Proposition \ref{L^i} and proof of Theorem \ref{h by betti} would prove (\ref{Kalai conj}) for closed homology manifolds.
   In dimension four with $\beta_2 = 0,$  (\ref{Kalai conj})  is equivalent to (\ref{walkup}).

   \section{Constructions} \label{Constructions}
   
   In order to completely characterize the $f$-vectors of all possible triangulations of a given space, we will need ways of constructing new triangulations from old ones which preserve homeomorphism type and alter the $f$-vector in a predictable fashion.  Two such techniques are bistellar moves and central retriangulations.  
   
   Let $F$ and $G$ be disjoint subsets of the vertices of $\Delta$ such that $|F|+|G|=d+1.$  Suppose that the vertex induced subcomplex of $\Delta$ with respect to $F \cup G$ is $F\ast \partial G.$   Removing $F \ast \partial G$ and replacing it with $\partial F \ast G$ is a  {\it( $|G|-1$)-bistellar move.} A 0-bistellar move is also called subdividing a  facet.  If $\Delta^\prime$ is obtained from $\Delta$ by a bistellar move, then $\Delta^\prime$ is homeomorphic to $\Delta$. The effect of a bistellar move on the $h$-vector is contained in the proposition below.  We omit its elementary proof.
   
   \begin{prop}
   Suppose $\Delta^\prime$ is obtained from $\Delta$ by an $m$-bistellar move.   Then
   $$h_i(\Delta^\prime) = 
      \begin{cases}
         h_i(\Delta) & i \le m \mbox{ or } i \ge d-m \\
         h_i(\Delta)+1 & m < i < d-m.
       \end{cases}$$
   \end{prop}

There are very few manifolds for which a complete characterization of all possible $f$-vectors of triangulations are known.  Aside from $S^1$ and closed $2$-manifolds \cite{JR} \cite{Ri}, the only  other spaces for which this question has been solved are $S^3, S^1 \times S^2, \R P^3,$ the nonorientable $S^2$ bundle over $S^1,$ and $S^4.$  The four $3$-manifolds and $S^4$ were done by Walkup \cite{Wal}.
   
   The first manifold we will consider in detail is $S^1 \times S^3.$  For every $n \ge 11,$ K\"{u}hnel and Lassmann constructed a vertex-transitive triangulation of $S^1 \times S^3$ with $n$ vertices and dihedral symmetry \cite{KL}.  We use $\Delta_{S^1 \times S^3}(n)$ to denote these complexes. (In \cite{KL} they used $M^4_3(n)$ for these triangulations.) Identify the vertices with $[n] = \{1,2,\dots,n\}.$  Since the triangulation is  invariant under the action of $\Z_n,$ it is sufficient to specify for which 4-tuples $(y_1,y_2,y_3,y_4)$ there are simplices of the form $\{x_1,x_2,x_3,x_4,x_5\},$ with $x_{i+1} - x_i = y_i \mod n, 1 \le i \le 4.$ The 4-tuples which generate $\Delta_{S^1 \times S^3}(n)$ are $[1,1,1,2], [1,1,2,1],[1,2,1,1],$ and $[2,1,1,1].$  The link of each vertex has 10 vertices, so $f_1 = 5n$ and $h_2 = 5n - 10 - 4(n-5) = n+10.$ So, $g_2= h_2 - h_1 = 15.$     By the generalized Dehn-Sommerville equations, the $g$-vector of a homology manifold without boundary determines its $h$-vector, and hence its $f$-vector.
\begin{thm}
   The following are equivalent.
   \begin{enumerate}
     \item[i.]
     $(1,g_1,g_2)$ is the $g$-vector of a triangulation of $S^1 \times S^3.$
    
     \item[ii.]
     $(1,g_1,g_2)$ is the $g$-vector of a closed four-dimensional  homology manifold $\Delta$ with $\beta_1 = 1$ and $\beta_2=0.$
     
     \item[iii.]
       $15 \le g_2 \le \binom{g_1+1}{2}.$
    \end{enumerate}
 \end{thm}

   \begin{proof}
    Evidently i $\to$ ii. {\it All} simplicial complexes satisfy $g_2 \le \binom{g_1+1}{2}.$ For $S^1 \times S^3$  with $\beta_1 = 1$, Theorem \ref{covering 2} implies $15 \le g_2$, hence ii $\to$ iii. 
   
   For iii $\to$ i, assume  $ h_2 \le \binom{h_1+1}{2}, h_2-h_1 \ge 15,$ and $h_1 \ge 6.$ As $h_1(\Delta_{S^1 \times S^3}(n)) = n-5$ and $h_2(\Delta_{S^1 \times S^3}(n)) = n+10,$ it is sufficient to show that for each $n$, beginning with $\Delta_{S^1 \times S^3}(n),$ it is possible to perform consecutive 1-bistellar moves, each such move adding exactly one edge to the 1-skeleton,  until the 1-skeleton is the complete graph on $n$ vertices. 
   
   What are the nonedges of $\Delta_{S^1 \times S^3}(n)?$  From the description of the facets, the link of $x \in [n]$ consists of all $y$ within $\pm 5 \mod n$ of $x.$ Hence $(x,y)$ is a nonedge of $\Delta_{S^1 \times S^3}(n)$ if and only if $x$ and $y$ are separated by at least $6$ modulo $n.$  
   
   First, group the nonedges of $\Delta_{S^1 \times S^3}(n)$ by the value of $y-x \mod n,$ where we insist  this difference be between $6$ and $n - 1.$   For instance, if $n = 14,$ then the first group contains the pairs $(1,7), (2,8), \dots, (7,13), \\ (8,14), (9,1), \dots, (14,6).$  Similarly, the second group would contain \\ $(1,8),(2,9),\dots,(6,13),(7,14), (8,1), \dots, (14,7).$  In general, if $n$ is odd, then the nonedges will partition into blocks each of which has cardinality $n,$ while if $n$ is even, the last block will have cardinality $n/2.$
   
   Starting with the first group, for each pair $(x, x+6)$ perform a 1-bistellar move using the facets $\{x, x+1, x+2, x+4,x+5\}$ and $\{x+1, x+2, x+4, x+5, x+6\}.$   Now consider the pairs $(x,x+7)$ in the second group.  From the bistellar move applied to the pair $(x,x+6), x$ is contained in a facet $\{x,x+1,x+2,x+5,x+6\},$ while the bistellar move applied to $(x+1,x+7)$ puts $x+7$ in the facet $\{x+1,x+2,x+5,x+6,x+7\}.$  Hence, we can now perform a 1-bistellar move for each pair in the second group.  Similarly, a bistellar move corresponding to $(x, x+8)$ in the third group can use the facets $\{x, x+1, x+2, x+6, x+7\}$ and $\{x+1, x+2, x+6,x+7,x+8\}$ obtained via the bistellar move from the second group.  Continuing in this way, it is possible to perform bistellar moves until the 1-skeleton is the complete graph on $n$ vertices.

   \end{proof}   
     
    An examination of the proof shows that after $kn$ 1-bistellar moves  the resulting complex once again has dihedral vertex-transitive symmetry.
     
     \begin{cor}
       If $n \ge 11, k \ge 5,$ and $kn \le \binom{n}{2},$ then there is a  vertex-transitive triangulation of $S^1 \times S^3$ with $n$ vertices, $kn$ edges and dihedral symmetry.
            \end{cor}
     
     When $n$ is odd, this is the best result possible since any vertex-transitive triangulation will have $kn$ edges for some $k.$ For even $n,$ any vertex-transitive triangulation must have $n+(kn/2)$ edges, and the catalog of small vertex-transitive triangulations by K\"ohler and Lutz \cite{KoLu} suggests that this may be possible once $n \ge 14.$  The manifolds $\Prefix^{4}{13}^2_1, \Prefix^{4}{13}^2_2,  \Prefix^{4}{14}^3_2, \Prefix^{4}{14}^3_3, \Prefix^{4}{15}^2_1, \Prefix^{4}{15}^2_2,$ and $\Prefix^{4}{15}^2_3$ in \cite{KoLu} come from the above construction.

     As indicated previously, the  complexes $\Delta_{S^1 \times S^3}(n)$ are part of a much larger family of triangulations introduced by K\"{u}hnel and Lassmann in \cite{KL}.  All of their spaces are disk or sphere bundles over tori.  They include a  collection $\Delta_{S^1 \times S^{2m-1}}(n),$ denoted by $M^{2m}_{2m-1}(n)$ in \cite{KL}, of triangulations of $S^1 \times S^{2m-1}$ for $m \ge 2$ and $n \ge 4m+3.$    When $n = 4m+3, \Delta_{S^1 \times S^{2m-1}}(n) = M^{2m},$ which we recall from Theorem \ref{min nonzero betti}, is a minimum triangulation of $S^1 \times S^{2m-1}.$ 
The construction of $\Delta_{S^1 \times S^{2m-1}}(n)$ is along the same lines as $\Delta_{S^1 \times S^3}(n).$ The facets are specified by the $2m$ difference vectors (modulo $n$) obtained by all possible permutations of  $[1,1, \dots, 1,2].$  The resulting complex has dihedral symmetry.  Using 1-bistellar moves in a fashion similar to above, and Theorem \ref{covering 2}, it is possible to prove the following.

\begin{prop}
Let $m \ge 2.$  There exists a triangulation of $S^1 \times S^{2m-1}$ with $n$ vertices and $e$ edges if and only if $n \ge 4m+3$ and $e - (2m+1) n \ge 0.$  If in addition $e$ is a multiple of $n,$ then there exists a triangulation which is vertex-transitive with dihedral symmetry.

\end{prop}
     
     A second technique for creating new triangulations out of old ones is the central retriangulation of a simple $(d-1)$-tree.  Let $B$ be a subcomplex of $\Delta$ which is a simplicial ball.  Remove all of the interior faces of $B$ and replace them with the interior faces of the cone on the boundary of $B,$ where the cone point is a new vertex. We call this new complex  the {\it central retriangulation} of $B.$ See Figure \ref{tree} for a simple example in dimension two.  If $\Delta^\prime$ is obtained from $\Delta$ by a retriangulation of $B,$ then $\Delta$ and $\Delta^\prime$ are homeomorphic. 

\begin{figure} 
 \scalebox{0.50}[0.50]{\includegraphics{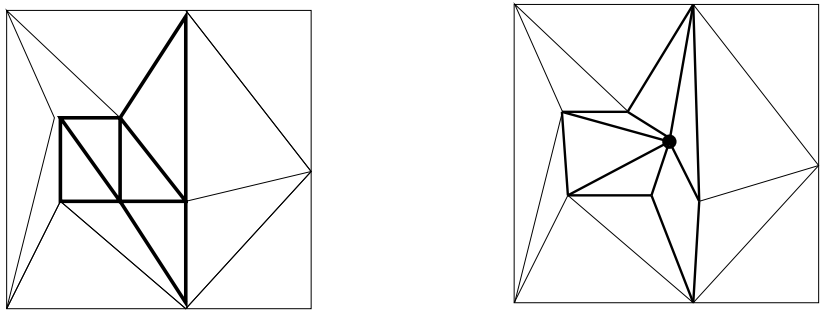}}
  \caption{Central retriangulation of a simple $2$-tree} \label{tree}
\end{figure}
     
    Let $T$ be a pure $(d-1)$-dimensional complex.  We say $T$ is a {\it simple $(d-1)$-tree} if the facets of $T$ can be ordered,  $\sigma_1, \dots, \sigma_m$ such that for each $i \ge 2$ the intersection of $\sigma_i$ with the union of all previous facets is a codimension one face of $\sigma_i$ which is on the boundary of $\cup^{i-1}_{j=1} \sigma_j.$  A simple 1-tree is a path.  The bold face subcomplex on the left hand side of Figure \ref{tree} is a simple 2-tree.  The {\it length} of a simple $(d-1)$-tree is the number facets in the tree.  Each facet, other than the first, adds exactly one new vertex to the tree.  An ordering of the vertices, $(v_1, \dots, v_{d+m-1}),$ of a simple $(d-1)$-tree $T$ is {\it natural} if there exists an ordering of the facets of $T$ such that the vertices of $\sigma_1$ are $(v_1,\dots,v_d)$ and for $i \ge 2,$ the new vertex introduced by $\sigma_i$ is $v_{d+i-1}.$  Any simple $(d-1)$-tree is a simplicial ball and its boundary is a stacked sphere.  
    
    \begin{prop} \label{h of crt}
      If $\Delta^\prime$ is obtained from $\Delta$ by central retriangulation of a simple $(d-1)$-tree of length $m,$ then $h_1(\Delta^\prime) = h_1(\Delta)+1,$ and $h_2(\Delta^\prime) = h_2(\Delta) + m.$
    \end{prop}
    
    The following idea is due to Walkup.  Indeed, our statement and proof are just the $(d-1)$-dimensional analogue of \cite[Lemma 7.3]{Wal}.    A simple $(d-1)$-tree in $\Delta$ is {\it spanning} if it contains all of the vertices of $\Delta.$ 
     \begin{prop}  \label{machine}
       Let $\Delta$ be a $2$-neighborly triangulation of a  homology manifold without boundary.  Suppose $\Delta$ contains a spanning $(d-1)$-tree $T$ such that every facet of $T$ contains a fixed set of distinct vertices $\{v_1,\dots,v_{d-3}\}.$ Equivalently, there is a codimension three face $\rho,$ and a spanning simple $2$-tree in the link of $\rho.$ Then for every pair $(a,b)$ satisfying $a \ge h_1(\Delta)$ and $g_2(\Delta) + a \le b \le \binom{a+1}{2}$ there exists a complex $\Delta^\prime$ which is homeomorphic to $\Delta$  with $h_1(\Delta^\prime) =a$ and $h_2(\Delta^\prime) = b.$
     \end{prop}
     
     \begin{proof}
     Let $\Delta = \Delta_0$ and let $\Delta_1$ be the complex obtained by a central retriangulation of $T$  and let $w_1$ be the new vertex in $\Delta_1.$ As $\Delta$ was neighborly and $T$ is spanning,  $\Delta^\prime$ is neighborly.  For each $i > d-3, \{v_1, \dots, v_{d-3}, v_i\}$ is a face of the boundary of $T.$  Hence $\{w_1,v_1, \dots, v_{d-3}, v_i\}$ is a face of $\Delta_1,$    so the link of $\{w_1,v_1,\dots,v_{d-3}\}$ contains all of the other vertices of $\Delta_1.$  Since $\Delta_1$ is a closed homology manifold, this link must be a circle.  This implies that $\Delta_1$ has a spanning simple $(d-1)$-tree $T_1$ and distinct vertices $\{v^\prime_1,\dots, v^\prime_{d-2}\}$ which are contained in every facet of $T_1.$  Repeating this process we obtain an infinite family of complexes $\Delta_k.$ For each $k,$ $\Delta_k$ is homeomorphic to $\Delta,$ 2-neighborly, and   $h_1(\Delta_k) = h_1(\Delta) + k.$    In addition, each $\Delta_k$ has a spanning simple $(d-1)$-tree $T_k$.   Given $j, 1 \le j \le h_1(\Delta)+k+1$, define $\Delta_{k,j}$ to be the complex obtained from $\Delta_k$ by a central retriangulation of the simple $(d-1)$-tree consisting of the the first $j$ facets     of $T_k.$  By Proposition \ref{h of crt} the collection of pairs $(h_1(\Delta_{k,j}), h_2(\Delta_{k,j})),$ is exactly the pairs $(a,b)$ guaranteed by the theorem.
          
     \end{proof}

\begin{thm} \label{cp2}
   The following are equivalent.
   \begin{enumerate}
     \item[i.]$(1,g_1,g_2)$ is the $g$-vector of a triangulation of $\C P^2.$
     
     \item[ii.]
      $(1,g_1,g_2)$ is the $g$-vector of a triangulation of a closed four-dimensional homology manifold with $\beta_1=0$ and  $\beta_2 = 1.$
     \item[iii.]
       $6 \le g_2 \le \binom{g_1+1}{2}.$
    \end{enumerate}
 \end{thm}

\begin{proof}
 As before, i $\to$ ii is trivial and  any complex satisfies $ g_2 \le \binom{g_1+1}{2}.$  Theorem \ref{even euler inequality} implies $6 \le g_2,$ and hence ii $\to$ iii.

For iii $\to$ i, the previous proposition  shows that it is sufficient to find a triangulation of $\C P^2$ with $g$-vector $(1,3,6)$ and simple $4$-tree which satisfies the conditions of the previous proposition.  Table \ref{CP2} contains such a triangulation, originally due to K\"{u}hnel \cite{KL2}, and Table \ref{cp2 tree} shows an appropriate simple $2$-tree in the link of an edge.

\end{proof}

\begin{table}
$$\begin{array}{llllc}
[1,2,3,4,5], & [1,2,3,4,7], & [1,2,3,5,8], & [1,2,3,7,8], \\
\left[1,2,4,5,6 \right], & [1,2,4,6,7], & [1,2,5,6,8], & [1,2,6,7,9],  \\
 \left[1,2,6,8,9 \right]  & [1,2,7,8,9] &[1,3,4,5,9],  & [1,3,4,7,8], \\ 
\left[1,3,4,8,9 \right], & [1,3,5,6,8], & [1,3,5,6,9], & [1,3,6,8,9],  \\
\left[1,4,5,6,7 \right], & [1,4,5,7,9], & [1,4,7,8,9], & [1,5,6,7,9], \\ \left[2,3,4,5,9 \right], & [2,3,4,6,7], & [2,3,4,6,9], & [2,3,5,7,8], \\ 
\left[2,3,5,7,9 \right], & [2,3,6,7,9], & [2,4,5,6,8], & [2,4,5,8,9], \\ \left[2,4,6,8,9\right], & [2,5,7,8,9], & [3,4,6,7,8], &[3,4,6,8,9], \\ \left[3,5,6,7,8 \right], & [3,5,6,7,9], & [4,5,6,7,8], & [4,5,7,8,9]
\end{array}$$
\caption{$\Delta(\C P^2)$ - A triangulation of $\C P^2$ with $h_1 = 4, h_2 = 10$ \cite{KL2}} \label{CP2}
\end{table}

\begin{table}
$$[3,4,7],[3,4,5],[4,5,6],[5,6,8],[6,8,9]$$
\caption{A spanning simple $2$-tree in the link of $[1,2]$ in $\Delta(\C P^2).$} \label{cp2 tree} 
\end{table}

 A similar technique can be used to characterize $g$-vectors of triangulations of $K3$ surfaces.  Any nonsingular quartic in $\C P^3$ is a {\it K3 surface.}  While different quartics can lead to distinct complex structures, all K3 surfaces are homeomorphic. In fact they are diffeomorphic, see, for instance \cite[Theorem 3.4.9]{GS}.   As a closed four manifold, every K3 surface is simply connected with second Betti number equal to twenty-two.  Hence any triangulation satisfies $h_3 - h_2 = 220 = \binom{12}{3}.$

\begin{thm}
The following are equivalent.
   \begin{enumerate}
    \item[i.]$(1,g_1,g_2)$ is the $g$-vector of a triangulation of  a K3 surface.
     
     \item[ii.]
      $(1,g_1,g_2)$ is the $g$-vector of a triangulation of a closed four dimensional homology manifold with $\beta_1 = 0, \beta_2 = 22.$
     \item[iii.]
       $ 55 \le g_2 \le \binom{g_1+1}{2}.$
    \end{enumerate}
\end{thm}   
     
\begin{proof}
  The proof of i $\to$ ii and ii $\to$ iii follows the same reasoning as in Theorem \ref{cp2}.  In addition, as in that proof, iii $\to$ i will be established with the existence of a triangulation $\Delta_{K3}$ of a K3 surface with $g$-vector $(1,10,55)$ and an appropriate simple $4$-tree.  Such a triangulation was given by Casella and K\"{u}hnel \cite{CK}.  We refer the reader to this reference for  the triangulation and a verification of its $g$-vector.  Figure \ref{k3 link},  which is \cite[Figure 1]{CK},  shows the link of an edge in this triangulation and contains a spanning simple $2$-tree in this link.  The labeled vertices on the boundary of the hexagons are identified and the numbering of the triangles is an ordering for the facets of the tree.  Coning this tree with the edge provides the desired $4$-tree.
\end{proof} 
    
\begin{figure} 
 \scalebox{0.50}[0.50]{\includegraphics{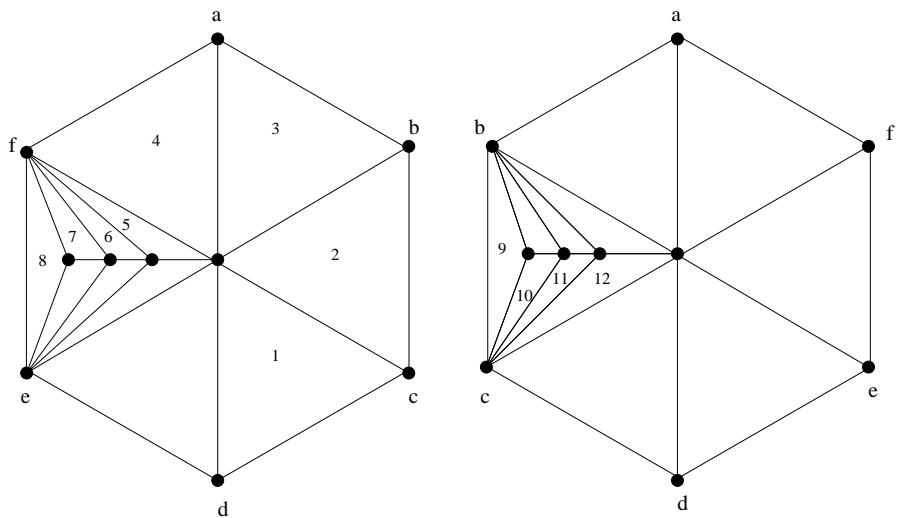}}
  \caption{Spanning $2$-tree in the link of an edge of $\Delta_{K3}$} \label{k3 link}
\end{figure}

\begin{thm}
The following are equivalent.
   \begin{enumerate}
     \item[i.]
     $(1,g_1,g_2)$ is the $g$-vector of a triangulation of the connected sum of $S^2 \times S^2$ with itself.
         \item[ii.]
      $(1,g_1,g_2)$ is the $g$-vector of a triangulation of a closed four-dimensional homology manifold with $\beta_1=0$ and $\beta_2=4.$

           \item[iii.]
       $  18 \le g_2 \le \binom{g_1+1}{2}.$
  \end{enumerate}
\end{thm}

\begin{proof}
Using the same logic as the previous two theorems, $g_2 \ge 18.$    Table \ref{s2xs2} shows  a triangulation with $g$-vector $(1,6,18)$.  This triangulation comes from \cite{Lu2}. Every pair of vertices in this triangulation span an edge except for $\{1,5\}, \{5,6\}$ and $\{1,6\}.$  For $g_2$ equal to 18,19 or 20, first perform 0,1 or 2 $1$-bistellar moves from Table \ref{3 moves}.  Then subdivide facets until the desired number of vertices are obtained.  As before, for $g_2 \ge 21$ Proposition \ref{machine} tells us that is it sufficient to find a spanning simple $2$-tree in the link of some edge of a $2$-neighborly triangulation of   $(S^2 \times S^2) \# (S^2 \times S^2)$ with $12$ vertices.  After the three $1$-bistellar moves in Table \ref{3 moves}, the triangulation in Table \ref{s2xs2} is $2$-neighborly.  Table \ref{s2xs2 tree} lists such a tree in the link of $[1,2].$

\end{proof}

\begin{table}
$$\begin{array}{llll}
\left[1,2,3,4,7\right], & [1,2,3,4,10], & [1,2,3,7,10], &[1,2,4,7,8], \\ \left[1,2,4,8,11\right], & [1,2,4,9,10], & [1,2,4,9,12], & [1,2,4,11,12], \\ 
\left[1,2,7,8,10\right], & [1,2,8,9,10], & [1,2,8,9,12], & [1,2,8,11,12], \\
\left[1,3,4,7,11 \right], & [1,3,4,10,11], & [1,3,7,8,11], & [1,3,7,8,12], \\ \left[1,3,7,9,10\right], & [1,3,7,9,12], &[1,3,8,11,12], & [1,3,9,10,12], \\ 
\left[1,3,10,11,12\right], & [1,4,7,8,11], & [1,4,9,10,11], & [1,4,9,11,12],\\
\left[1,7,8,9,10\right], & [1,7,8,9,12], & [1,9,10,11,12], & [2,3,4,6,9], \\
\left[2,3,4,6,10\right], & [2,3,4,7,12], &[2,3,4,9,12], & [2,3,5,7,9], \\ 
\left[2,3,5,7,10\right], & [2,3,5,8,10], & [2,3,5,8,11], & [2,3,5,9,11], \\
\left[2,3,6,9,11\right], & [2,3,6,10,11], & [2,3,7,9,12], &[2,3,8,10,11], \\ 
\left[2,4,5,7,8\right], & [2,4,5,7,12], &[2,4,5,8,11], & [2,4,5,11,12], \\ 
\left[2,4,6,9,10\right], & [2,5,7,8,10], & [2,5,7,9,11], & [2,5,7,11,12], \\
\left[2,6,7,9,11\right], & [2,6,7,9,12], & [2,6,7,11,12], &[2,6,8,9,10], \\ 
\left[2,6,8,9,12\right], & [2,6,8,10,12], &[2,6,10,11,12], &[2,8,10,11,12],\\ 
\left[3,4,5,8,9\right], & [3,4,5,8,12], & [3,4,5,9,12], & [3,4,6,7,11], \\
\left[3,4,6,7,12\right], & [3,4,6,8,9], & [3,4,6,8,12], &[3,4,6,10,11], \\ 
\left[3,5,7,9,10\right],& [3,5,8,9,11], &[3,5,8,10,12], &[3,5,9,10,12], \\ 
\left[3,6,7,8,11\right], & [3,6,7,8,12], & [3,6,8,9,11], & [3,8,10,11,12], \\
\left[4,5,7,8,10\right], & [4,5,7,10,12], & [4,5,8,9,11], &[4,5,8,10,12],  \\ 
\left[4,5,9,11,12\right], & [4,6,7,10,11], &[4,6,7,10,12], &[4,6,8,9,10], \\ 
\left[4,6,8,10,12\right], & [4,7,8,9,10], & [4,7,8,9,11], & [4,7,9,10,11],\\
\left[5,7,9,10,11\right],  & [5,7,10,11,12], &[5,9,10,11,12], &[6,7,8,9,11],\\ \left[6,7,8,9,12\right],& [6,7,10,11,12] & &
  \end{array}
  $$
\caption{$\Delta((S^2 \times S^2) \# (S^2 \times S^2))$ - A triangulation of $(S^2 \times S^2) \# (S^2 \times S^2)$ with $h_1 = 7$ and $g_2 = 18$ \cite{Lu2}.}  \label{s2xs2}
\end{table}

\begin{table}
$$([1,2,3,7,10],[2,3,5,7,10])$$
$$([2,3,5,9,11],[2,3,6,9,11])$$
$$([1,2,4,9,10],[2,4,6,9,10])$$

\caption{Three 1-bistellar moves on $\Delta((S^2 \times S^2) \# (S^2 \times S^2))$}  \label{3 moves}
\end{table}

\begin{table}
$$ \{[3,5,10], [5,7,10], [7,8,10], [8,9,10], [8,9,12], [4,9,12], [4,6,9], [4,11,12]\}$$
\caption{A spanning simple 2-tree in the link of $\{1,2\}$ in $\Delta((S^2 \times S^2) \# (S^2 \times S^2))$ after three 1-bistellar moves.} \label{s2xs2 tree}
\end{table}

While the methods we have introduced are not sufficient to completely characterize the $h$-vectors of higher dimensional spaces, many partial results are possible.  For instance, consider $S^3 \times S^3.$  

\begin{prop}  
  The component-wise minimum $h$-vector for triangulations of $S^3 \times S^3$ is $(1,6,21,56,126,-21,20,-1).$   There exists a triangulation $\Delta$ of $S^3 \times S^3$ with $h_1(\Delta) = a$ and $h_2(\Delta) = b$ if and only if $ 15 \le b-a \le \binom{a}{2}.$

\end{prop} 
 
\begin{proof}
Let $\Delta$ be a triangulation of $S^3 \times S^3.$ By Schenzel's formula (see Theorem \ref{schen}) $h_i = h^\prime_i$ for $i=1,2,3$ or $4.$ The generalized Dehn-Sommerville equations imply that $h^\prime_4 - h^\prime_3 = 70 = \binom{8}{4}.$  Hence, $h^\prime_4 = h_4 \ge \binom{9}{4}, h^\prime_3 = h_3 \ge \binom{8}{3}, h^\prime_2 = h_2 \ge \binom{7}{2}$ and $h^\prime_1 = h_1 \ge \binom{6}{1}.$  In \cite[Section 7]{Lu}, Lutz gives a triangulation of $S^3 \times S^3$ with this $h$-vector.  In addition, the link of the face denoted by $[1,2,3,4]$ in that triangulation has a spanning simple $2$-tree, so Proposition \ref{machine} applies.

\end{proof}

In \cite{Wal} Walkup proved that for any closed three manifold $M$ there exists $\gamma(M)$ such that for any pair $(a,b)$ with $\gamma(M) \le b \le \binom{a+1}{2},$ there exists a triangulation $\Delta$ of $M$ with $g_1(\Delta) = a$ and $g_2(\Delta) = b.$  In fact, this is true for all closed homology manifolds of dimension at least three which can be triangulated.  We prove this in a series of lemmas which are an adaptation of Walkup's proof to higher dimensions.  

\begin{lem}  \label{lem 1}
  Let $M$ be a connected homology manifold without boundary of dimension $d-1$.  If $M$ has a triangulation, then $M$ has a triangulation which contains a spanning simple $(d-1)$-tree.
\end{lem}

\begin{proof}
Let $\Delta$ be a triangulation of $T.$ Since $\Delta$ is connected, there exists a simple $(d-1)$-tree $T$ and a dimension preserving simplicial map $\phi:T \to \Delta$ which maps surjectively onto the  vertices of $\Delta.$  Indeed, $T$ and $\phi$ can be constructed inductively by beginning with a facet of $\Delta$ and attaching new facets to $T$ along codimension one faces corresponding to free codimension one faces of the image of $T$ until all of the vertices of $\Delta$ are in the image of $\phi.$ 

If $\phi$ is one-to-one on the vertices of $T,$ then the image of $T$ satisfies the lemma.  Otherwise, let $y_1,\dots,y_s$ be a natural ordering of the vertices of $T.$  Let $y_t$ be the last vertex of $T$ such that $|\phi^{-1}(\phi(y))| \ge 2.$   The definition of $y_t$ implies that $\phi$ is a simplicial isomorphism when restricted to the closed star of $y_t.$  Hence $B,$ the image of $\overline{\st} (y_t),$ is a ball.  Now let $\Delta^\prime$ be the complex obtained from $\Delta$ by the central retriangulation of  $B.$  Define a new  map $\phi^\prime$ by $\phi^\prime(y_i) = \phi(y_i),$ except $\phi(y_t) = w,$ where $w$ is the new vertex of $\Delta^\prime.$  To see that $\phi^\prime$ induces a simplicial map, it is sufficient to note that if $\phi(y_t) \cup \rho$ is a facet of $\Delta$ which includes $\phi(y_t),$ then $\rho$ is a face on the boundary of $B,$ so $w \cup \rho$ is a face of $\Delta^\prime.$  Thus, $\phi^\prime: T \to \Delta^\prime$ is a simplicial map which also maps surjectively onto the vertices of $\Delta^\prime$, is one-to-one on all the vertices after $y_t,$ and $|\phi^{\prime^{-1}}(\phi^\prime(y_t))|=|\phi^{-1}(\phi^\prime(y_t))|-1.$ Repeating this procedure enough times gives the desired triangulation and $(d-1)$-tree.  See Figure \ref{retriangulate} for a two-dimensional portrayal.
\end{proof}

\begin{figure} 
 \scalebox{0.80}[0.80]{\includegraphics{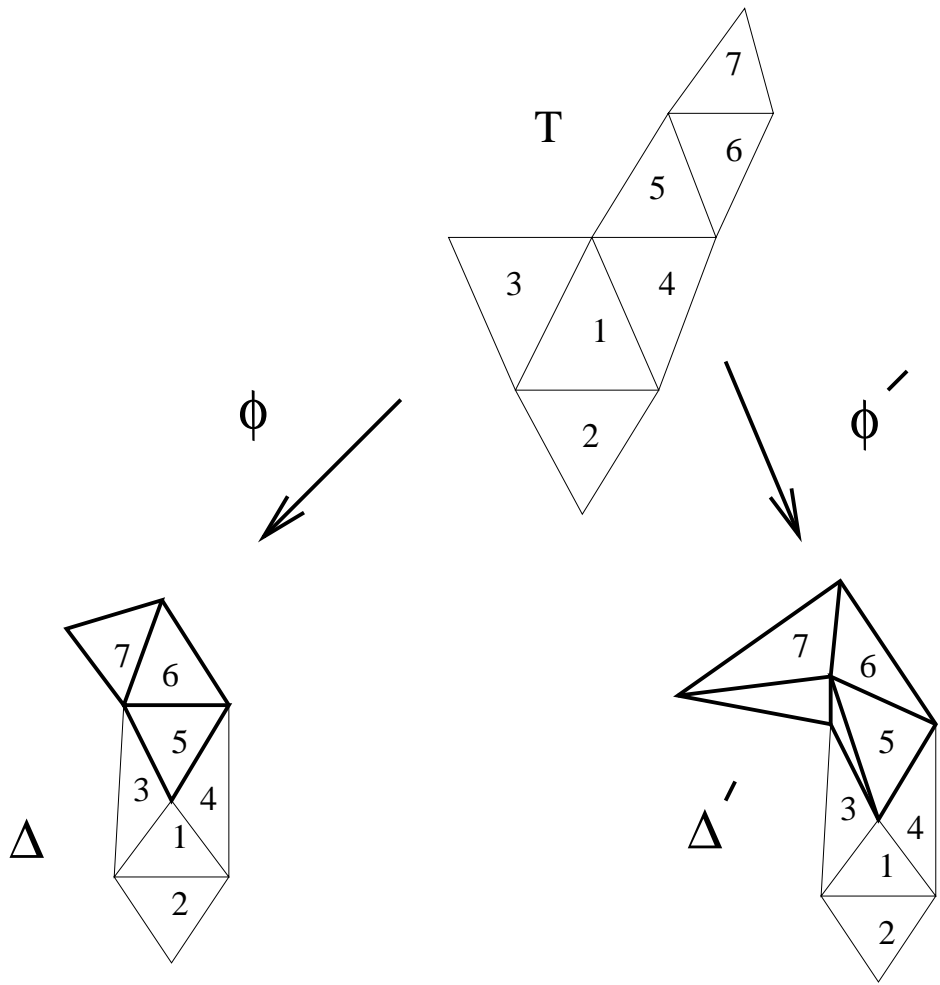}}
  \caption{Creating a spanning simple $(d-1)$-tree} \label{retriangulate}
\end{figure}

\begin{lem}  \label{lem 2}
  Let $M$ be a connected homology manifold without boundary of dimension $d-1$.  If $M$ has a triangulation, then $M$ has a triangulation $\Delta$ which contains a spanning simple $(d-1)$-tree $T$  in which every facet of $T$ contains  a fixed  $(d-3)$-dimensional face.    

\end{lem}

\begin{proof}
By the previous lemma we can choose a triangulation $\Delta_0$ of $M$ and a spanning simple $(d-1)$-tree $T_0$ of $\Delta_0.$  Let $\Delta^\prime_1$ be the complex obtained from the central retriangulation of $T_0$ and let $w_1$ be the new vertex. The link of $w_1$ contains all of the other vertices of $\Delta^\prime_1.$   Now we repeat the procedure used in the proof of the previous lemma.  However, the original tree $T$ and $\phi$ are constructed by first choosing a simple $(d-2)$-tree $\tilde{T}_1$ and simplicial map $\tilde{\phi}_1: \tilde{T}_1 \to \lk w_1$ such that all of the vertices in the link of $w_1$ are in the  image of $\tilde{\phi}_1.$ Then let $T_1 = \{w_1\} \ast \tilde{T}_1$ and $\phi_1$ be $\tilde{\phi}_1$ extended to $T_1$ by setting $\phi_1(w_1) = w_1.$ Now each facet of the image of $\phi_1$ contains $w_1$ and $|\phi^{-1}_1(w_1)| = 1.$  Hence, while retriangulating $\Delta^\prime_1$ for the purposes of forcing $\phi_1$ to be one-to-one on the vertices, $w_1$ will remain in all of the facets in the image.   At the end of this process we will have a triangulation $\Delta_1$ and spanning simple $(d-1)$-tree $T_1$ all of whose facets contain $w_1.$ If $d=4$ we are done.  Otherwise, let $\Delta^\prime_2$ be the complex obtained from $\Delta_1$ by the central retriangulation of $T_1$ and let $w_2$ be the new vertex.  Repeat this  process beginning with a simple $(d-3)$-tree $\tilde{T}_2$ and a vertex spanning simplicial map $\tilde{\phi}_2$ from $\tilde{T}_2$ into  the link of $\{w_1,w_2\}.$  Arguing as before we will end up with a simple $(d-1)$ tree $T_2$ in $\Delta_2$ all of whose facets contain $\{w_1,w_2\}.$  This process can be repeated until we obtain the promised triangulation and spanning simple $(d-1)$-tree.
\end{proof}

\begin{lem}  \label{lem 3}
Let $M$ be a connected homology manifold without boundary of dimension $d-1$.  If $M$ has a triangulation, then $M$ has a $2-$neighborly triangulation which contains a spanning simple $(d-1)$-tree $T$ and a codimension three face $\sigma$ such that $\sigma$ is in every facet of $T.$ 

\end{lem}

\begin{proof}
Using the previous lemma, we begin with a triangulation $\Delta_0$ and spanning simple $(d-1)$-tree $T_0$ every facet of which contains the codimension three face $\rho_0.$   If $\Delta_0$ is 2-neighborly we are done.  So suppose $\Delta_0$ has $m$ pairs of vertices which do not have an edge between them.  Let $\Delta_1$ be the complex obtained by the central retriangulation of $T_0$ and let $w_1$ be the new vertex.  As in the proof of Proposition \ref{machine}, the link of $\rho_0 \cup \{w_1\}$ contains all of the other vertices of $\Delta_1.$ So $\Delta_1$  has the same $m$ pairs of vertices without edges.  Induction will complete the proof if we can construct a triangulation $\Delta$ and face $\rho \in \Delta$ which satisfy the hypothesis of the lemma such that $\Delta$ has only $m-1$ pairs of vertices with no edge between them.

Let $x,y$ be vertices in $\Delta_1$ with no edge between them.  Since $\rho_0 \cup \{w_1\}$ has codimension two, its link is a circle.  As $x$ and $y$ do not have an edge between them they must be in the link of $\rho_0 \cup \{w_1\}$ and be separated by at least one vertex as one travels around the circle. 
The construction of $\Delta$ and $\rho$ consists of three steps.

Step 1:  Retriangulate so that $x$ and $y$ are only separated by one vertex in the link of a codimension two face whose link contains all of the other vertices. 
Write the vertices in the link of $\rho_0 \cup w_1$ in cyclic order, $(x, v_1, \dots, v_s, y, u_1, \dots, u_t).$ If either $s$ or $t$ is one, then step one is complete.  If not, let $P_1$ be the path with ordered vertices $(x, v_1,\dots,v_s,y)$ and let $P_2$ be the path with ordered vertices $(v_s,y,u_1,\dots,u_t).$  Choose $w_0 \in \rho_0$ and set $\rho^\prime_0$ to be $\rho_0$ with $w_0$ removed. If $d=4,$ then $\rho^\prime_0 = \emptyset.$ Set $S = ((\{w_1\} \cup \rho^\prime_0) \ast P_1) \cup ((\{w_0\} \cup \rho^\prime_0) \ast P_2)$ and $T = (\rho_0 \cup w_1) \ast (P_1 \cup P_2).$  Now $T$ is a spanning simple $(d-1)$-tree in $\Delta_1$ and $S$ is a spanning simple $(d-2)$-tree in the boundary of $T.$
Every facet of $S$ contains $\rho^\prime_0$.  Figure \ref{S} shows $S$ in the link of $\rho^\prime_0.$

\begin{figure} 
 \scalebox{0.80}[0.80]{\includegraphics{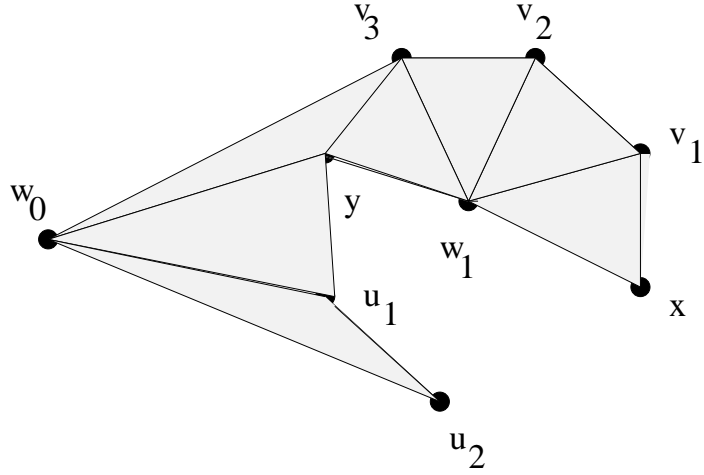}}
  \caption{The tree $S$ in the link of $\rho^\prime_0$} \label{S}
\end{figure}

Let $\Delta_2$ be the complex obtained from $\Delta_1$ by the central retriangulation of $T$ and let $w_2$ be the new vertex.  In this complex $S \ast \{w_2\}$ is a spanning simple $(d-1)$-tree.  So now we let $\Delta_3$ be the complex obtained from $\Delta_2$ by the central retriangulation of $S \ast \{w_2\}$ and call the new vertex $w_3.$  Using Figure \ref{S}, we can see that the link of $\{w_2,w_3\} \ast \rho^\prime_0$ in $\Delta_3$ is the circle $(y, w_1,x,v_1,\dots,v_s,w_0,u_t,u_{t-1},\dots,u_1).$  Note that in each of the retriangulations the pairs of vertices without edges have not changed.  

Step 2:  Perform a 1-bistellar move on $\partial(\{x,y\}) \ast (\{w_1,w_2,w_3\} \ast \rho^\prime_0)$ and call the resulting complex $\Delta_4.$  This introduces an edge between $x$ and $y$ and leaves $m-1$ pairs of vertices without edges.  However, $w_1$ is no longer in the link of $\{w_2,w_3\} \ast \rho^\prime_0.$

Step 3:  Let $P_4$ be the path whose vertices in order are $(x,y,u_1,\dots,u_t, \\ w_0,v_s,\dots,v_1).$  Form a new $(d-2)$-tree, $\tilde{S}_4 = \rho^\prime_0 \ast  ((\{w_2\} \ast P_4) \cup \{x, w_1, y\} ).$  Now $T_4 = \{w_3\} \ast \tilde{S}_4$ is a spanning simple $(d-1)$-tree in $\Delta_4.$  In addition, $S_4 = \tilde{S}_4 \cup \{\{y, w_1, w_3\} \ast \rho^\prime_0\}$ is a spanning simple $(d-2)$-tree in the boundary of $T_4.$  Two more triangulations will finish the job.  First, $\Delta_5$ is the complex obtained from $\Delta_4$ by the central retriangulation of $T_4$ with new vertex $w_4.$  In this complex $T_5=\{w_4\} \ast S_4$ is a spanning simple $(d-1)$-tree. Finally, set $\Delta$ to be the complex obtained from the central retriangulation of $T_5$ with new vertex $w_5,$ and let $\rho = \{w_4,w_5 \} \cup \rho^\prime_0.$  Here we see that the link of the codimension two face $\rho$ contains all the other vertices of $\Delta$ and there are only $m-1$ pairs of vertices that do not span an edge.

\end{proof}

\begin{thm}
Let $M$ be a connected  homology manifold without boundary of dimension $d-1$.   Then there exists $\gamma(M)$ such that for every pair $(a,b)$ with $\gamma(M) \le b \le \binom{a+1}{2}$ there exists a triangulation $\Delta$ of $M$ with $g_1(\Delta) = a$ and $g_2(\Delta) = b.$  

\end{thm}

\begin{proof}
  Apply Proposition \ref{machine} to the triangulation guaranteed by Lemma \ref{lem 3}.

\end{proof}

\begin{cor}
  Let $M$ be a closed homology manifold of dimension at least three.  If $M$ has a triangulation, then  for $n$ sufficiently large there exist $2$-neighborly triangulations of $M$ with $n$ vertices.
\end{cor}

As mentioned before, our  proof of the existence of 2-neighborly triangulations is based on Walkup's proof of this for  3-manifolds \cite{Wal}.  A very different approach in dimension three is Sarkaria's  proof of the existence of 2-neighborly triangulations of 3-manifolds with or without boundary \cite{Sa}.

{\it Acknowledgements:}  The author had helpful suggestions and conversations  with several people, including Anders Bj\"{o}rner, Louis Billera, Frank Lutz, Gil Kalai, Richard Ehrenborg,  Eran Nevo, Isabella Novik, Branden Owens, and G\"unter Ziegler.

\end{document}